\documentclass{amsart}
\usepackage{amsmath,amsthm}
\usepackage{amsfonts,amssymb}
\numberwithin{equation}{section}
\allowdisplaybreaks[3]
\usepackage{enumerate}
\usepackage{tikz}
\usepackage{cite} 
\usepackage{lscape}
\usepackage{microtype}
\usepackage{graphicx}
\usepackage[english]{babel}
\usepackage{fancyhdr}

\newtheorem{theorem}{Theorem}

\numberwithin{theorem}{section}

\theoremstyle{definition}
\newtheorem{definition}[theorem]{Definition}

\theoremstyle{remark}
\newtheorem{remark}[theorem]{Remark}

\newcommand{\cF}{\mathcal{F}}

\newcommand{\mf}[1]{\mathfrak{#1}} 

\newcommand{\ux}{\underline{x}}

\newcommand{\uy}{\underline{y}}

\newcommand{\uz}{\underline{z}}

\newcommand{\upx}{\partial_{\underline{x}}}

\newcommand{\br}[2]{\genfrac{}{}{0pt}{}{#1}{#2}}

\begin{document}
\title[On the radially deformed Fourier transform]{On the radially deformed Fourier transform}

\author{Hendrik De Bie}
\address{Department of Electronics and Information Systems \\Faculty of Engineering and Architecture\\Ghent University\\Krijgslaan 281, 9000 Gent\\ Belgium.}
\email{Hendrik.DeBie@UGent.be}
	
\author{Ze Yang}
\address{Department of Electronics and Information Systems \\Faculty of Engineering and Architecture \\Ghent University\\Krijgslaan 281, 9000 Gent\\ Belgium.}
\email{zeyang.cn@outlook.com}

\date{\today}

\begin{abstract}
In this paper we consider the kernel of the radially deformed Fourier transform introduced in the context of Clifford analysis in \cite{DBOSS-2013}. By adapting the Laplace transform method from \cite{DDBP-2018}, we obtain the Laplace domain expressions of the kernel for the cases of $m=2$ and $m > 2$ when $1+c=\frac{1}{n}, n\in \mathbb{N}_0\backslash\{1\}$ with $n$ odd. Moreover, we show that the expressions can be simplified using the Poisson kernel and the generating function of the Gegenbauer polynomials. As a consequence, the inverse formulas are used to get the integral expressions of the kernel in terms of Mittag-Leffler functions.
\end{abstract}
	
\maketitle
	

\section{Introduction}
Clifford analysis \cite{FRF, CASF, GM-1991} is a refinement of harmonic analysis in the $m$-dimensional Euclidean space $\mathbb{R}^m$. The orthogonal Clifford algebra $\mathcal{C}\ell_{0, m}$ over $\mathbb{R}^m$ is generated by the canonical basis $e_i$, $i=1,\dots , m$. These generators $e_i$ satisfy the multiplication rules $e_ie_j+e_je_i = -2\delta_{ij},\ 1\le i,\,j\le m$. The $\mf{sl}_2$ algebra generated by the basic operators in  harmonic analysis is refined to the Lie superalgebra $\mf{osp}(1|2)$. This is achieved by introducing the Dirac operator $\upx = \sum_{j=1}^{m} \partial_{x_{j}} e_{j}$ and its Fourier symbol $\ux=\sum_{i=1}^{m}e_ix_i$. These operators satisfy $\ux^2=-\langle x,x \rangle = -\sum_{j=1}^mx_j^2=-|x|^2$ and $\upx^2 = -\Delta$.

The radially deformed Fourier transform
\begin{equation}\label{rddofour}
\cF_{D} = e^{i\frac{\pi}{2}\left( \frac{1}{2}+\frac{m-1}{2(1+c)}\right)}e^{\frac{-i\pi}{4(1+c)^2}\left({\bf D}^2 - (1+c)^2\ux^2\right)}
\end{equation}
where $L={\bf D}^2-(1+c)^2\ux^2$ is the generalized Hamiltonian, was initially introduced as the Fourier transform associated with the so-called radially deformed Dirac operator
\[
{\bf D}:=\partial_{\ux}+c|x|^{-2}\ux\,\mathbb{E},\quad c>-1
\]
where $\mathbb{E}=\sum_{i=1}^{m}x_{i}\partial_{x_{i}}$ the Euler operator and we have put $a = 2, b = 0$ in  the operator studied in \cite{DBOSS-2012}. The investigation of \eqref{rddofour} was continued in \cite{DBOSS-2013} by a group theoretical approach. A series expansion for the kernel of the integral representation of the deformed Fourier transform 
\begin{equation}\label{rdftker}
\cF_{D} \left( f \right) \left(y\right) = \frac{\Gamma{\left(\frac{m}{2}\right)}}{2 \pi^{m/2}} \, \int_{\mathbb{R}^m} K_{m}^c(x,y)\, f(x)\, h(r)\, \mathrm{d}x
\end{equation}
with $h(r)=r^{1-\frac{1+mc}{1+c}}$ was determined explicitly, and the recursion relations for the kernel were shown subsequently. Explicit formulas for the kernel of the Fourier transform were obtained when the dimension is even and $1+c=\frac{1}{n}, n\in \mathbb{N}_0\backslash\{1\}$ with $n$ odd in \cite{DBDSE-2017}, where $\mathbb{N}_0$ is the set of non-negative integers.

More recently, Laplace transform methods (see \cite{Const, DDBP-2018}) have started to play an important role in the study of the kernels of hypercomplex Fourier transforms \cite{FTCA, DBR1, DBR2}. In \cite{DDBP-2018}, the method was developed to obtain explicit expressions for the kernel of the $(\kappa,\,a)$-generalized Fourier transform for $\kappa =0$. This was obtained by the introduction of an auxiliary variable in the series expansion of the kernel and subsequently taking the Laplace transform. By means of the Poisson kernel, the Laplace domain expression of the kernel can be simplified. Finally, using the Prabhakar function, the explicit expressions of the kernel were given via Laplace inversions. 

In this paper, we consider the kernel of the radially deformed Fourier transform \eqref{rddofour} in both even and odd dimensions by adapting the Laplace transform method in \cite{DDBP-2018}, for the case of $1+c=\frac{1}{n}, n\in \mathbb{N}_0\backslash\{1\}$ with $n$ odd. By introducing an auxiliary variable $t$ in the kernel, we can take the Laplace transform in $t$. In dimension 2, we begin by considering the kernel shown in Section 5.3 in \cite{DBDSE-2017}. The Laplace transform of the kernel can be simplified using two relations for infinite sums of trigonometric functions. In dimension $m > 2$, the Laplace domain expression of kernel can be obtained by means of the Poisson kernel and the generating function of the Gegenbauer polynomial. Furthermore, these results allow us to derive the explicit expressions for the kernel in terms of the Mittag-Leffler functions when $m =0$ and $m >2$ respectively. The structure of the Laplace transform method is shown in the following figure, where we introduce the notation $\lambda = (m-2)/2$:

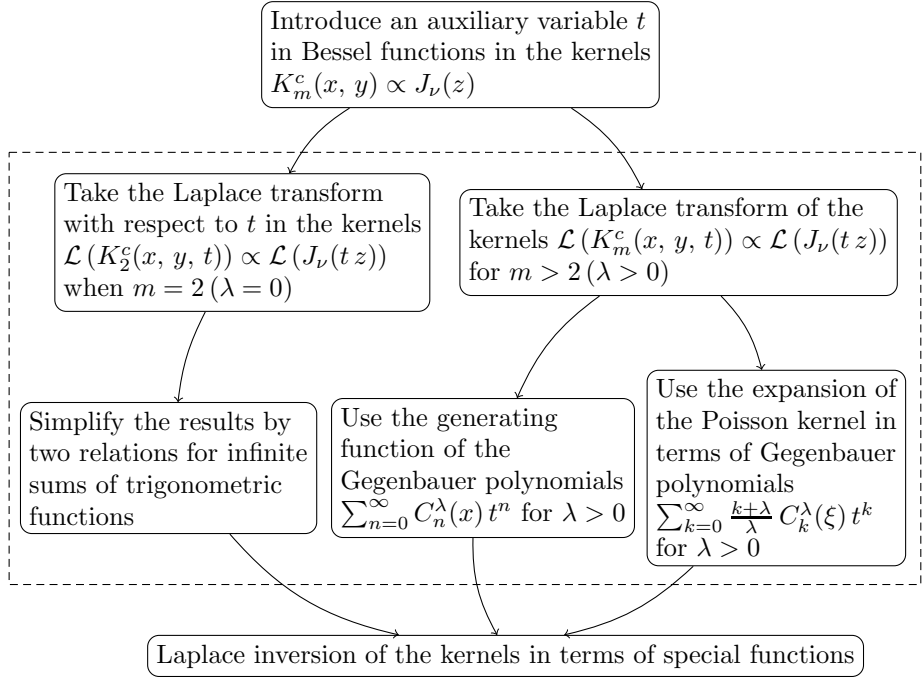
\begin{figure}[h]\centering
\begin{tikzpicture}
\node
[rectangle,draw, rounded corners=1.0ex
  ] (2) at(10.4,0)[draw, align=left] {Introduce an auxiliary variable $t$\\ in Bessel functions in the kernels \\$K^c_m(x,\,y)\propto J_{\nu}(z)$};
\node
[rectangle,draw, rounded corners=1.0ex
  ] (4) at(7.5,-2.5)[draw, align=left] {Take the Laplace transform\\ with respect to $t$ in the kernels\\ $\mathcal{L}\left(K^c_2(x,\,y,\,t)\right)\propto \mathcal{L}\left(J_{\nu}(t\,z)\right)$\\
  when $m=2\, (\lambda =0)$};
\node
[rectangle,draw, rounded corners=1.0ex
  ] (5) at(13.3,-2.5)[draw, align=left] {Take the Laplace transform of the\\kernels $\mathcal{L}\left(K^c_m(x,\,y,\,t)\right)\propto \mathcal{L}\left(J_{\nu}(t\,z)\right)$\\for $m>2\,(\lambda >0)$};
\node
[rectangle,draw, rounded corners=1.0ex
] (7) at(10.7,-5.5)[draw, align=left] {Use the generating\\ function of the\\ Gegenbauer polynomials\\ $\sum_{n=0}^{\infty}C_n^{\lambda}(x)\,t^n$ for $\lambda>0$};
\node
[rectangle,draw, rounded corners=1.0ex
] (8) at(14.6,-5.5)[draw, align=left] {Use the expansion of\\ the Poisson kernel in\\ terms of Gegenbauer\\ polynomials\\ $\sum_{k=0}^{\infty}\frac{k+\lambda}{\lambda}\,C_k^{\lambda}(\xi)\,t^k$\\ for $\lambda>0$};
\node
[rectangle,draw, rounded corners=1.0ex] (9) at(6.52,-5.5) [draw, align=left] {Simplify the results by\\ two relations for infinite\\ sums of trigonometric\\functions};
\node
[rectangle,draw, rounded corners=1.0ex] (12) at(11,-8)[draw, align=left] {Laplace inversion of the kernels in terms of special functions by setting $t=1$};
\draw[->] (2)  to [bend right = 15] (4);
\draw[->] (2)  to [bend left = 15] (5);
\draw[->] (5)  to [bend right = 15] (7);
\draw[->] (5)  to [bend left = 15] (8);
\draw[->] (4)  to [bend right = 10] (9);
\draw[->] (9)  to [bend right = 18] (12);
\draw[->] (7)  to [bend right = 15] (12);
\draw[->] (8)  to [bend left = 15] (12);
\draw [densely dashed] (4.4,-1.3) to (16.52,-1.3) to (16.52,-7.05) to (4.4,-7.05) to (4.4,-1.3) ;
\end{tikzpicture}
\begin{center}
\begin{minipage}{6.0in}
\begin{itemize}
\item
The notation $\propto$ means that the equations hold up to constant factors.
\item The expressions of the kernel in the Laplace domain are indicated by the\\ dashed box.
\end{itemize}
\end{minipage}
\end{center}
\end{figure}\label{conceptmap}
This paper is organized as follows. In Section \ref{preliminary} we introduce  basic facts on the radially deformed Fourier transform and some results necessary for the sequel. In Section \ref{KerofRDFTinDim2lapDom} we obtain the kernel of the Fourier transform in dimension 2 in the Laplace domain. In Section \ref{Generalevenker} we derive the Laplace transform of the kernel when $m> 2$. Finally, the integral expressions of the kernel in terms of the Mittag-Leffler functions are given in Section \ref{MINLeffexPofkernel}. Conclusions can be found at the end of this paper.
\section{Preliminaries}\label{preliminary}
In this section, we give a brief overview of the radially deformed Fourier transform, special functions and the Laplace transform.

In \cite{DBOSS-2013}, the kernel $K_{m}^c$ of the deformed Fourier transform \eqref{rdftker} is given by 
\begin{equation}\label{p29}
K_{m}^c = \frac{1}{2\lambda}z^{-\frac{\mu -2}{2}} A_{\lambda}+ \frac{1}{2}z^{-\frac{\mu -2}{2}}B_{\lambda}-z^{-\frac{\mu }{2}}(\ux \wedge \uy)\,C_{\lambda}
\end{equation}
with
\begin{align}\label{kernels}
\begin{split}
A_{\lambda} &= \sum_{k=0}^{+\infty} (k+\lambda)\,\left( \alpha_k\, {J}_{\frac{\gamma_k}{2}-1}(z)-i\, \alpha_{k-1}\, {J}_{\frac{\gamma_{k-1}}{2}}(z)\right)C_k^{\lambda}(w),\\
B_{\lambda} &=  \sum_{k=0}^{+\infty}\left( \alpha_k\, {J}_{\frac{\gamma_k}{2}-1}(z)+i\, \alpha_{k-1}\, {J}_{\frac{\gamma_{k-1}}{2}}(z)\right)C_k^{\lambda}(w),\\
C_{\lambda} &=  \sum_{k=1}^{+\infty} \left( \alpha_k\, {J}_{\frac{\gamma_k}{2}-1}(z)+i\, \alpha_{k-1}\, {J}_{\frac{\gamma_{k-1}}{2}}(z)\right)C_{k-1}^{\lambda+1}(w),
\end{split}
\end{align}
where $\ux \wedge \uy:=\sum_{j<k}e_je_k\,(x_j\,y_k-x_k\,y_j)$ is the wedge product of the vectors $\ux$ and $\uy, \, \uz = |x||y|, \, w = \frac{\langle x, \, y \rangle}{z}, \, \lambda = \frac{m-2}{2},\, \mu = 1+\frac{m-1}{1+c}, \, \alpha_k = e^{-\frac{i\pi k}{2(1+c)}}, \, \alpha_{-1}=0, \, \gamma_k= \frac{2}{1+c}\left(k+\frac{m-2}{2}\right)+\frac{c+2}{1+c}$ and $c>-1$. ${J}_{\nu}$ denotes the Bessel function and $C_k^{\lambda}$ the Gegenbauer polynomial.

In \cite{DBDSE-2017}, the authors obtained explicit formulas for the above kernel when the dimension is even and $1+c=\frac{1}{n}, n\in \mathbb{N}_0\backslash\{1\}$ with $n$ odd by adapting the method developed in \cite{HDeB-2013}.

When $m=2$, the kernel was given explicitly in the following theorem:
\begin{theorem}\label{DBDSEker2}\cite{DBDSE-2017}
If $m=2\,(\lambda =0)$ and $1+\frac{1}{c} = \frac{1}{n}, n\in \mathbb{N}_0\backslash\{1\}$ with $n$ odd, then the kernel of the deformed Fourier transform $\mathcal{F}_{D}= e^{i\frac{\pi}{2}\left( \frac{1}{2}+\frac{m-1}{2(1+c)}\right)}e^{\frac{-i\pi}{4(1+c)^2}\left({\bf D}^2 - (1+c)^2\ux^2\right)}$ takes the form
\begin{align*}
K_2^c &= \frac{1}{n}i^{\frac{n-1}{2}}\bigg\{ z^{\frac{1-n}{2}}\sum_{j=0}^{n-1}\cos{\left(\left(\frac{n-1}{2}\right)\left(\frac{\theta +2\pi j}{n}\right)\right)e^{-i z \cos{\left( \frac{\theta +2\pi j}{n}\right)}}} \\
&\quad + \frac{1}{\sin{\theta}}z^{-\frac{n+1}{2}}(\ux\wedge\uy)\sum_{j=0}^{n-1}\sin{\left(\left(\frac{n-1}{2}\right)\left(\frac{\theta +2\pi j}{n}\right)\right)e^{-i z \cos{ \left( \frac{\theta +2\pi j}{n}\right)}}}\bigg\}
\end{align*}
with $\theta = \arccos{w}, \, z=|x||y|, \, w = \frac{\langle x, y \rangle}{z}$.
\end{theorem}
The kernel in even dimensions $m>2$ was derived as iterated derivatives of Theorem \ref{DBDSEker2} using the recursion relations obtained in \cite{DBOSS-2013}.
These iterated derivatives of the kernel can in principle be computed recursively. 

For $n\in\mathbb{N}$ and $\lambda >-1/2$, the Gegenbauer polynomials are defined as 
\[
C_n^{\lambda}(w)=\sum_{j=0}^{\left\lfloor  \frac{n}{2} \right\rfloor}(-1)^j\frac{\Gamma (n-j-\lambda)}{\Gamma (\lambda)\,j!\,(n-2j)!}\,(2w)^{n-2j}.
\]
Next, we need the following expansion of the Poisson kernel for the unit ball in terms of Gegenbauer polynomials.
\begin{theorem} \cite{DX-2014}\label{Poistheo}
For $x, y\in\mathbb{R}^m$ and $|y|\le |x|=1$, the Poisson kernel for the unit ball is 
\begin{align*}
P(x,y)&=\frac{1-|y|^2}{|x-y|^m}\\
&=\frac{1-|y|^2}{\left(1-2\xi |y|+|y|^2\right)^{m/2}}\\
&=\sum_{j=0}^{\infty}\frac{j+m/2-1}{m/2-1}\,C_j^{m/2-1}(\xi)\,|y|^j, \quad \xi=\langle x, \frac{y}{|y|}\rangle.
\end{align*}
This result can be extended for $\lambda >0$, we have
\begin{align}\label{analy1}
\frac{1-|y|^2}{\left(1-2\xi |y|+|y|^2\right)^{\lambda +1}}=\sum_{j=0}^{\infty}\frac{j+
\lambda}{\lambda}\,C_j^{\lambda}(\xi)\,|y|^j.
\end{align}
It is still valid for $z\in \mathbb{C}, \,|z|<1$ and $|\xi|<1,$ (see \cite{Hyper})
\begin{align}\label{analy2}
\frac{1-z^2}{\left(1-2\xi z+z^2\right)^{\lambda +1}}=\sum_{j=0}^{\infty}\frac{j+
\lambda}{\lambda}\,C_j^{\lambda}(\xi)\,z^j.
\end{align}
\end{theorem}
\begin{remark}\label{Poissoncontinuation}
The validity of the analytic continuation of \eqref{analy1} to \eqref{analy2} for the whole unit disk has been proved in \cite{DDBP-2018}.
\end{remark}
The generating function of the Gegenbauer polynomials (see e.g. \cite{EMOT-1953}, p.177 (29)) is given by
\begin{equation}\label{GeneFuncGegen}
\frac{1}{(1-2xt+t^2)^{\alpha}}=\sum_{n=0}^{\infty}C_n^{(\alpha)}(x)\,t^n 
\end{equation}
for $0\le |x|<1,\, |t|\le1,\, \alpha >0$.

Suppose that $f$ is a real or complex valued function of the variable $t>0$ and $s$ is a complex parameter. The (one-sided) Laplace transform of $f$ is defined as (see e.g. \cite{Laplace}) 
\begin{equation}\label{10}
F(s)=\mathcal{L}\left(f(t)\right)=\int_0^\infty{e^{-st}f(t)\,\mathrm{d}t}.
\end{equation}
By Lerch's theorem, if we restrict to continuous functions on $[0, \infty)$, the inverse transform is uniquely determined
$\mathcal{L}^{-1}\left(F(s)\right)=f(t).$

In the paper, we will need the following Laplace transform formula from \cite{AE-1954}:
\begin{equation}\label{exponlapfor}
\mathcal{L}(e^{-\alpha t})=\frac{1}{s+\alpha}, \quad {\rm Re}\,s >{\rm Re}\,\alpha.
\end{equation}
The Bessel function of the first kind is given by the infinite power series
\begin{equation}\label{p2}
J_\nu(x) = \sum_{n=0}^{\infty} \frac{(-1)^n}{n!\ \Gamma (n+\nu+1)} \left ( \frac{x}{2} \right )^{2n+\nu}.
\end{equation}
For ${\rm Re}\,s>|{\rm Im}\, b|$, we have the Laplace transform of \eqref{p2}
\begin{equation}\label{LapBessfor}
\mathcal{L}\left({J}_{\nu}(b\,t)\right)=\frac{1}{\sqrt{s^2+b^2}}\left(\frac{b}{s+\sqrt{s^2+b^2}}\right)^{\nu},\quad {\rm Re}\,\nu>-1.
\end{equation}

We further introduce definitions of two Mittag-Leffler functions in \cite{MH-2008, GKMR-2014}.
\begin{definition}
\label{defML}
The two-parametric Mittag-Leffler function is defined by
\begin{equation}\label{parbmitlef}
E_{\alpha, \beta}(z):=\sum_{n=0}^{\infty}\frac{z^n}{\Gamma(\alpha n +\beta)}, \quad \alpha, \beta \in \mathbb{C}, {\rm Re}\,\alpha >0, {\rm Re}\,\beta >0.
\end{equation}
\end{definition}
\begin{definition}
\label{defPML}
The Prabhakar generalized Mittag-Leffler function is
 \begin{equation}\label{PRABmittleff}
E^{\delta}_{\alpha, \beta}(z):=\sum^{\infty}_{n=0}\frac{(\delta)_n\, z^n}{\Gamma(\alpha n+\beta)\,n!},
\end{equation}
where $\alpha, \beta, \delta\in\mathbb{C}$ with ${\rm Re}\,\alpha >0$. 
\end{definition}
\begin{remark}\label{remark2variML}
When $\delta =1$, we have $E^{1}_{\alpha, \beta}(z)=E_{\alpha, \beta}(z)$ (see e.g. \cite{MH-2008}, (2.3.3)).
\end{remark}
The Laplace transform of the 
Prabhakar function \eqref{PRABmittleff} is
\begin{equation}\label{LaPPrabhaMLe}
\mathcal{L}\left(t^{\beta -1}E^{\delta}_{\alpha, \beta} (b\,t^{\alpha})\right)=\frac{1}{s^{\beta}}\frac{1}{(1-b\,s^{-\alpha})^{\delta}},
\end{equation}
where ${\rm Re}\,\alpha >0, {\rm Re}\,\beta>0, {\rm Re}\,s>0$ and $s>|b|^{1/({\rm Re}\,\alpha)}$, see \cite{MH-2008}.\\
We will also need the following inverse transform in \cite{PBM-1992}, for ${\rm Re}\,\nu>-1,\  {\rm Re}\,s>|{\rm Im}\, a|$,  
\begin{equation}\label{MLinvers}
\begin{aligned}
&\quad\mathcal{L}^{-1}\left(\frac{\left(\sqrt{s^2+a^2}-s\right)^{\nu}}{\sqrt{s^2+a^2}}F\left(\sqrt{s^2+a^2}-s\right)\right)\\
&=(a^{2}\,t)^{\nu /2}\int_{0}^{t}(t+2\tau)^{-\nu/2}J_{\nu}(a\sqrt{t^2+2\tau t})\,f(\tau)\,\mathrm{d}\tau,
\end{aligned}
\end{equation}
where $\mathcal{L}(f(t))=F(s)$.
We also need 
the convolution formula of the Laplace transform. Denoting $\mathcal{L}(g(t))=G(s)$ and $\mathcal{L}(f(t))=F(s)$, we have 
\begin{equation}\label{convoluLap}
\mathcal{L}^{-1}\left(F(s)\,G(s)\right)=\int_0^t f(\tau)\,g(t-\tau)\,\mathrm{d}\tau.
\end{equation}
\section{The kernel in the Laplace domain}
\subsection{The dimension $m=2$}\label{KerofRDFTinDim2lapDom}
In \cite{DBDSE-2017}, when $m=2$, i.e. $\lambda =0$, the kernel of the deformed Fourier transform in case of $1+c=\frac{1}{n}, n\in\mathbb{N}_0\backslash \{1\}$ with $n$ odd and $n=2\tilde{n}+1, \tilde{n}\in\mathbb{N}_0$ takes the form
\begin{align*}
K_{2}^c &= \frac{1}{2}z^{\frac{c}{2(1+c)}}\lim_{\lambda\to 0}\frac{1}{\lambda}A_{\lambda}+ \frac{1}{2}z^{\frac{c}{2(1+c)}}\lim_{\lambda\to 0}B_{\lambda}-z^{-\frac{2+c}{2(1+c)}}(\ux \wedge \uy)\,\lim_{\lambda\to 0}C_{\lambda}
\end{align*}
with
\begin{align*}
\lim_{\lambda\to 0}\frac{1}{\lambda}A_{\lambda}&=-J_{\tilde{n}}(z)
+2\sum_{k=0}^{+\infty}(-i)^{kn}J_{kn+\tilde{n}}(z)\cos{(k\theta)}\\
&\quad +2(-1)^{\tilde{n}}\sum_{k=1}^{+\infty}(-i)^{kn}J_{kn-\tilde{n}}(z)\cos{(k\theta)}, \quad w =\cos{\theta};\\
\lim_{\lambda\to 0}B_{\lambda}&=J_{\tilde{n}}(z);\\
\lim_{\lambda\to 0}C_{\lambda}&=\frac{1}{\sin{\theta}}\sum_{k=1}^{+\infty}(-i)^{kn}J_{kn+\tilde{n}}(z)\sin{(k\theta)}\\
&\quad -(-1)^{\tilde{n}}\frac{1}{\sin{\theta}}\sum_{k=1}^{+\infty}(-i)^{kn}J_{kn-\tilde{n}}(z)\sin{(k\theta)}, \quad w =\cos{\theta}.
\end{align*}

Inspired by the Laplace transform method in \cite{DDBP-2018}, we introduce an auxiliary variable $t$ in the Bessel functions of the kernel.
For the scalar part of $K_2^c$, we have 
\begin{equation}\label{tinscaker}
\begin{aligned}
K_{2,scal}^c
&= \frac{1}{2}z^{\frac{c}{2(1+c)}}\left(\lim_{\lambda\to 0}\frac{1}{\lambda}A_{\lambda}+\lim_{\lambda\to 0}B_{\lambda}\right)\\
&=z^{-\tilde{n}}\sum_{k=0}^{+\infty}(-i)^{kn}J_{kn+\tilde{n}}(z\, t)\cos{(k\theta)}\\
&\quad +(-1)^{\tilde{n}}z^{-\tilde{n}}\sum_{k=1}^{+\infty}(-i)^{kn}J_{kn-\tilde{n}}(z\, t)\cos{(k\theta)},
\end{aligned}
\end{equation}
and for the bivector part 
\begin{equation}\label{tinbivker}
\begin{aligned}
 K_{2,biv}^c
 &=-z^{-\frac{2+c}{2(1+c)}}\lim_{\lambda\to 0}C_{\lambda}\\
 &=-\frac{1}{\sin{\theta}}z^{-\tilde{n}-1}\sum_{k=1}^{+\infty}(-i)^{kn}J_{kn+\tilde{n}}(z\, t)\sin{(k\theta)}\\
 &\quad +(-1)^{\tilde{n}}\frac{1}{\sin{\theta}}z^{-\tilde{n}-1}\sum_{k=1}^{+\infty}(-i)^{kn}J_{kn-\tilde{n}}(z\, t)\sin{(k\theta)}.   
\end{aligned}
\end{equation}
Next, let us first consider the scalar part of the kernel. For ${\rm Re}\,s$ big enough, we take the Laplace transform of the scalar part with respect to $t$ by formula \eqref{LapBessfor}, which yields
\begin{align*}
\mathcal{L}\left(K_{2,scal}^c\right)
&=\frac{1}{r}\left(\frac{1}{s+r}\right)^{\tilde{n}}\sum_{k=0}^{+\infty}(-i)^{kn}\left(\frac{z}{s+r}\right)^{kn}\cos{(k\theta)}\\
&\quad + (-1)^{\tilde{n}}z^{-2\tilde{n}}\frac{1}{r}\left(\frac{1}{s+r}\right)^{-\tilde{n}}\sum_{k=1}^{+\infty}(-i)^{kn}\left(\frac{z}{s+r}\right)^{kn}\cos{(k\theta)}\\
&=\frac{1}{r}\left(\frac{1}{s+r}\right)^{\tilde{n}}\sum_{k=0}^{+\infty}(-i)^{kn}\left(\frac{z}{s+r}\right)^{kn}\cos{(k\theta)}\\
&\quad + (-1)^{\tilde{n}}z^{-2\tilde{n}}\frac{1}{r}\left(\frac{1}{s+r}\right)^{-\tilde{n}}\sum_{k=0}^{+\infty}(-i)^{kn}\left(\frac{z}{s+r}\right)^{kn}\cos{(k\theta)}\\
&\quad - (-1)^{\tilde{n}}z^{-2\tilde{n}}\frac{1}{r}\left(\frac{1}{s+r}\right)^{-\tilde{n}}
\end{align*}
with $r=\sqrt{s^2+z^2}$, $\uz = |x||y|$, $w = \frac{\langle x, \, y \rangle}{z}$. The validity of transforming term by term in \eqref{tinscaker} is guaranteed by the following theorem.
\begin{theorem}\label{laptranguarantee}\cite{GD-2012}
Let the function $F(s)$ be represented by a series of $\mathcal{L}$-transforms 
\[
F(s)=\sum_{v=0}^{\infty}F_v(s), \quad F_v(s)=\mathcal{L}(f_v(t)),
\]
where all integrals
\[
\mathcal{L}(f_v)=\int_0^{\infty}e^{-st}\,f_v(t)\,\mathrm{d}t =F_v(s), \quad (v=0,1,\cdots)
\]
converge in a common half-plane ${\rm Re}\,s \ge x_0$. Moreover, we require that the integrals
\[
\mathcal{L}(|f_v|)=\int_{0}^{\infty}e^{-st}|f_v(t)|\,\mathrm{d}t=G_v, \quad (v=0,1,\cdots)
\]
and the series
\[
\sum_{v=0}^{\infty} G_v(x_0)
\]
converge which implies that $\sum_{v=0}^{\infty} F_v(s)$ converges absolutely and uniformly in ${\rm Re}\,s \ge x_0$. Then $\sum_{v=0}^{\infty} f_v(t)$ converges, absolutely, towards a function $f(t)$ for all $t\ge 0$; this $f(t)$ is the original function of $F(s)$;
\[
\mathcal{L}\left(\sum_{v=0}^{\infty} f_v(t)\right)=\sum_{v=0}^{\infty} F_v(s).
\]
\end{theorem}
Using the well-known relation (see for e.g. \cite{SLL5-2009}, p. 136 (21.41)):
\begin{equation}\label{sumprocuctCOS}
\sum_{k=0}^{\infty}r^k \cos{kx}=\frac{1-r\cos{x}}{1-2r\cos{x}+r^2},\quad |r|<1,
\end{equation}
we obtain the following expression for the Laplace transform of the scalar part of $K_2^c$
\begin{equation}\label{scakersr}
\begin{aligned}
\mathcal{L}\left(K_{2,scal}^c\right)
&=\frac{1}{r}\left(\frac{1}{s+r}\right)^{\tilde{n}}\frac{1-w u_R}{1-2w u_R+u_R^2}\\
&\quad +(-1)^{\tilde{n}}z^{-2\tilde{n}}\frac{1}{r}\left(\frac{1}{s+r}\right)^{-\tilde{n}}\frac{1-w u_R}{1-2w u_R+u_R^2}\\
&\quad - (-1)^{\tilde{n}}z^{-2\tilde{n}}\frac{1}{r}\left(\frac{1}{s+r}\right)^{-\tilde{n}}\\
&=\frac{1}{r}\left(\frac{1}{s+r}\right)^{\tilde{n}}\frac{1}{1-2w u_R+u_R^2}\\
&\quad -\frac{1}{r}\left(\frac{1}{s+r}\right)^{\tilde{n}}\frac{w u_R}{1-2w u_R+u_R^2}\\
&\quad +(-1)^{\tilde{n}}z^{-2\tilde{n}}\frac{1}{r}\left(\frac{1}{s+r}\right)^{-\tilde{n}}\frac{w u_R}{1-2w u_R+u_R^2}\\
&\quad -(-1)^{\tilde{n}}z^{-2\tilde{n}}\frac{1}{r}\left(\frac{1}{s+r}\right)^{-\tilde{n}}\frac{u_R^2}{1-2w u_R+u_R^2}
\end{aligned}
\end{equation}
where $u_R=\left(\frac{-iz}{s+r}\right)^n$.
Simplifying the results yields
\begin{align*}
\mathcal{L}\left(K_{2,scal}^c\right)
&=\frac{1}{r}\frac{(s+r)^{n-\tilde{n}}+(-1)^{\tilde{n}}(r-s)^{n-\tilde{n}}}
{{(s+r)^n-2w (-iz)^n+(-1)^{n}(r-s)^n}}\\
&\quad -w (-i)^n\frac{z}{r}\frac{(r-s)^{\tilde{n}}-(-1)^{n}(s+r)^{\tilde{n}}}
{{(s+r)^n-2w (-iz)^n+(-1)^{n}(r-s)^n}}.
\end{align*}
Similarly, we can compute the Laplace transform of the bivector part
\begin{align*}
\mathcal{L}\left(K_{2,biv}^c\right)
&=-\frac{1}{\sin{\theta}}z^{-1}\frac{1}{r}\left(\frac{1}{s+r}\right)^{\tilde{n}}\sum_{k=1}^{+\infty}(-i)^{kn}\left(\frac{z}{s+r}\right)^{kn}\sin{(k\theta)}\\
 &\quad +(-1)^{\tilde{n}}\frac{1}{\sin{\theta}}z^{-2\tilde{n}-1}\frac{1}{r}\left(\frac{1}{s+r}\right)^{-\tilde{n}}\sum_{k=0}^{+\infty}(-i)^{kn}\left(\frac{z}{s+r}\right)^{kn}\sin{(k\theta)}.
\end{align*}
By the relation (see \cite{SLL5-2009}, p. 136 (21.42))
\begin{equation}\label{sumprocuctSIN}
\sum_{k=1}^{\infty}r^k \sin{kx}=\frac{r\sin{x}}{1-2r\cos{x}+r^2},\quad |r|<1,
\end{equation}
it follows that
\begin{equation}\label{bivkersr}
\begin{aligned}
\mathcal{L}\left(K_{2,biv}^c\right)
&=-\frac{1}{zr}\left(\frac{1}{s+r}\right)^{\tilde{n}}\frac{u_R}{1-2w u_R+u_R^2}\\
&\quad +(-1)^{\tilde{n}}z^{-2\tilde{n}}\frac{1}{zr}\left(\frac{1}{s+r}\right)^{-\tilde{n}}\frac{u_R}{1-2w u_R+u_R^2}
\end{aligned}
\end{equation}
where $u_R=\left(\frac{-iz}{s+r}\right)^n$. \\
By direct computation, we obtain
\begin{align*}
\mathcal{L}\left(K_{2,biv}^c\right)
&=-(-i)^{n}\frac{1}{r}\frac{(r-s)^{\tilde{n}}-(-1)^{n}(s+r)^{\tilde{n}}}{(s+r)^n-2w (-iz)^n+(-1)^{n}(r-s)^n}\\
&=\frac{i}{r}\frac{(-1)^{\tilde{n}}(r-s)^{\tilde{n}}-(s+r)^{\tilde{n}}}
{{(s+r)^n-2w (-iz)^n+(-1)^n(r-s)^n}},
\end{align*}
since 
\begin{equation}\label{srrsz2}
(s+r)^{\tilde{n}}(r-s)^{\tilde{n}}
=(r^2-s^2)^{\tilde{n}}
=\left(\left(\sqrt{s^2+z^2}\right)^2-s^2\right)^{\tilde{n}}
=z^{2\tilde{n}}.
\end{equation}
Hence we can summarize the results in the following theorem.
\begin{theorem}\label{kerneldim2}
If $m=2$ $(\lambda =0)$, $1+c=\frac{1}{n}$, $n\in \mathbb{N}_0\backslash\ \{1\}$ with $n$ odd, i.e. $n=2\tilde{n}+1$, $\tilde{n}\in\mathbb{N}_0$, then the kernel of the deformed Fourier transform $\cF_{D} =   e^{i\frac{\pi}{2}\left( \frac{1}{2}+\frac{m-1}{2(1+c)}\right)}e^{\frac{-i\pi}{4(1+c)^2}\left({\bf D}^2 - (1+c)^2\ux^2\right)}$ in the Laplace domain is given by 
\begin{align*}
\mathcal{L}\left(K_{2}^c\right)&=\frac{1}{r}\frac{(s+r)^{\tilde{n}+1}+(-1)^{\tilde{n}}(r-s)^{\tilde{n}+1}}
{{(s+r)^n-2w (-iz)^n+(-1)^{n}(r-s)^n}}\\
&\quad +w z\frac{i}{r}\frac{(-1)^{\tilde{n}}(r-s)^{\tilde{n}}-(s+r)^{\tilde{n}}}
{{(s+r)^n-2w (-iz)^n+(-1)^{n}(r-s)^n}}\\
&\quad +(\ux \wedge \uy)\,\frac{i}{r}\frac{(-1)^{\tilde{n}}(r-s)^{\tilde{n}}-(s+r)^{\tilde{n}}}
{{(s+r)^n-2w (-iz)^n+(-1)^n(r-s)^n}}
\end{align*}
where $r=\sqrt{s^2+z^2},\, z=|x||y|,\, w = \frac{\langle x, \, y \rangle}{z}$.
\end{theorem}
By adapting the factorization for the polynomials utilized in Lemma 1 in \cite{DDBP-2018}, we rewrite Theorem \ref{kerneldim2} in the following theorem.
\begin{theorem}\label{Ker2coroll}
If the dimension $m=2\, (\lambda=0)$ and $1+c=\frac{1}{n}$, $n\in \mathbb{N}_0\backslash\ \{1\}$ odd, $n=2\tilde{n}+1$, $\tilde{n}\in\mathbb{N}_0$, the following expansions hold:\\
\emph{\underline{Case 1}: $\tilde{n}$ odd \, $(\tilde{n}=1,3,5,\dots)$} 
\begin{align*}
\mathcal{L}\left(K_{2}^c\right)&=\frac{1}{2^{\tilde{n}}}\frac{\prod_{l=0,\,l\ne\frac{\tilde{n}+1}{2}}^{\tilde{n}}\left(s-iz\sin{\left(\frac{l\pi}{\tilde{n}+1}\right)}\right)}
{{\prod_{l=0}^{n-1}\left(s+iz\cos{\left(\frac{\theta+2\pi l}{n}\right)}\right)}}\\
&\quad -\frac{i\,\omega\,z}{2^{\tilde{n}+1}}\frac{\prod_{l=1}^{\tilde{n}-1}\left(s-iz\cos{\left(\frac{l\pi}{\tilde{n}}\right)}\right)}
{{\prod_{l=0}^{n-1}\left(s+iz\cos{\left(\frac{\theta+2\pi l}{n}\right)}\right)}}\\
&\quad 
 -\frac{i}{2^{\tilde{n}+1}}(\ux \wedge \uy)\,\frac{\,\prod_{l=1}^{\tilde{n}-1}\left(s-iz\cos{\left(\frac{l\pi}{\tilde{n}}\right)}\right)}
{{\prod_{l=0}^{n-1}\left(s+i z\cos{\left(\frac{\theta+2\pi l}{n}\right)}\right)}};
\end{align*}
\emph{\underline{Case 2}: $\tilde{n}$ even \, $(\tilde{n}=0,2,4,\dots)$} 
\begin{align*}
\mathcal{L}\left(K_{2}^c\right)&=\frac{1}{2^{\tilde{n}}}\frac{\prod_{l=1}^{\tilde{n}}\left(s-iz\cos{\left(\frac{l\pi}{\tilde{n}+1}\right)}\right)}
{{\prod_{l=0}^{n-1}\left(s+iz\cos{\left(\frac{\theta+2\pi l}{n}\right)}\right)}}\\
&\quad -\frac{i\,\omega\,z}{2^{\tilde{n}+1}}\frac{\prod_{l=0,\,l\ne\frac{\tilde{n}}{2}}^{\tilde{n}-1}\left(s-iz\sin{\left(\frac{l\pi}{\tilde{n}}\right)}\right)}
{{\prod_{l=0}^{n-1}\left(s+iz\cos{\left(\frac{\theta+2\pi l}{n}\right)}\right)}}\\
&\quad 
 -\frac{i}{2^{\tilde{n}+1}}(\ux \wedge \uy)\,\frac{\prod_{l=0,\,l\ne\frac{\tilde{n}}{2}}^{\tilde{n}-1}\left(s-iz\sin{\left(\frac{l\pi}{\tilde{n}}\right)}\right)}
{{\prod_{l=0}^{n-1}\left(s+iz\cos{\left(\frac{\theta+2\pi l}{n}\right)}\right)}}
\end{align*}
where $\theta=\arccos{w},\, w = \frac{\langle x, \, y \rangle}{z},\, z=|x||y|$.
\end{theorem}
\begin{proof}
We only prove the results for the case $\tilde{n}$ odd. 
Let us define the denominator and numerator polynomials of the kernel in Theorem \ref{kerneldim2} as 
\begin{align*}
P_n(s)&:=(s+r)^n-2\omega (-iz)^n+(-1)^n(r-s)^n,\\
Q_{\tilde{n}}(s)&:=\frac{(s+r)^{\tilde{n}+1}+(-1)^{\tilde{n}}(r-s)^{\tilde{n}+1}}{r},\\
O_{\tilde{n}-1}(s)&:=\frac{(-1)^{\tilde{n}}(r-s)^{\tilde{n}}-(s+r)^{\tilde{n}}}{r}.
\end{align*}
\begin{enumerate}
\item We first consider the denominator $P_n(s)$,
\begin{align*}
P_n(s)
&=(s+r)^n-2\omega (-iz)^n+(-1)^n\,(r-s)^n\\
&=\sum_{k=0}^n\binom{n}{k}s^{n-k}r^k-2\omega (-iz)^n +(-1)^n\sum_{k=0}^n\binom{n}{k}(-1)^{n-k}\,s^{n-k}r^k\\
&=\left(\sum_{k=0}^n\binom{n}{k}s^{n-k}r^k\,(1+(-1)^k)\right)-2\omega (-iz)^n\\
&=2\sum_{k=0}^{\lfloor \frac{n}{2} \rfloor}\binom{n}{2k}s^{n-2k}\,(s^2+z^2)^k-2\omega (-iz)^n.
\end{align*} 
This means that $P_n(s)$ is a polynomial of degree $n$ in $s$. The coefficient of $s^n$ is $2\sum_{k=0}^{\lfloor \frac{n}{2} \rfloor}\binom{n}{2k}=2^n$.\\
Next, we verify $P_n(s_l)=0$ with $s_l=-iz\cos{\left(\frac{\theta+2\pi l}{n}\right)}, \,l=0, \cdots, n-1$. Denote $w = \cos{(\theta)}=\frac{e^{i\theta}+e^{-i\theta}}{2}$. When $\sin{\left(\frac{\theta +2\pi l}{n} \right)}\ge 0$, we have
\begin{align*}
P_n(s_l)
&=(s_l+r_l)^n-2\omega (-iz)^n+(-1)^n\,(r_l-s_l)^n\\
&=(-iz)^n\bigg[\left(\cos{\left(\frac{\theta+2\pi l}{n}\right)}+i\sin{\left(\frac{\theta+2\pi l}{n}\right)}\right)^n\\
&\quad +\left(\cos{\left(\frac{\theta +2\pi l}{n}\right)}-i\sin{\left(\frac{\theta+2\pi l}{n}\right)}\right)^n-2\omega\bigg]\\
&=(-iz)^n\left(e^{i\theta}-2\left(\frac{e^{i\theta}+e^{-i\theta}}{2}\right)+e^{-i\theta}\right)\\
&=0.
\end{align*}
where in the second step we have used Euler's formula.
When $\sin{\left(\frac{\theta +2\pi l}{n} \right)}< 0$ we have $P_n(s_j)=0$ by a similar calculation. \\
Therefore, $s_l, l=0,\ldots, n-1$ are all roots of $P_n$ and we get the factorization
\begin{align*}
P_n(s)
&=2\sum_{k=0}^{\lfloor \frac{n}{2} \rfloor}\binom{n}{2k}\,\prod_{l=0}^{n-1}\left(s+iz\cos{\left(\frac{\theta+2\pi l}{n}\right)}\right)\\
&=2^n\,\prod_{l=0}^{n-1}\left(s+iz\cos{\left(\frac{\theta+2\pi l}{n}\right)}\right).
\end{align*}
\item For $Q_{\tilde{n}}(s)$, we first claim that $Q_{\tilde{n}}(s)$ is a polynomial of degree $\tilde{n}$ in $s$. Indeed, we have
\begin{align*}
Q_{\tilde{n}}(s)&=\frac{(s+r)^{\tilde{n}+1}+(-1)^{\tilde{n}}(r-s)^{\tilde{n}+1}}{r}\\
&=\frac{1}{r}\bigg[\sum_{k=0}^{\tilde{n}+1}\binom{\tilde{n}+1}{k}s^{\tilde{n}+1-k}\,r^k\\
&\quad +(-1)^{\tilde{n}}\sum_{k=0}^{\tilde{n}
+1}\binom{\tilde{n}+1}{k}(-1)^{\tilde{n}+1-k}\,s^{\tilde{n}+1-k}\,r^k\bigg]\\ 
&=\frac{1}{r}\sum_{k=0}^{\tilde{n}+1}\binom{\tilde{n}+1}{k}\,s^{\tilde{n}+1-k}\,r^k\,(1-(-1)^k)\\
&=2\sum_{k=0}^{\lfloor \frac{\tilde{n}+1}{2} \rfloor}\binom{\tilde{n}+1}{2k+1}\,s^{\tilde{n}-2k}\,(s^2+z^2)^k.
\end{align*}
The coefficient of $s^{\tilde{n}}$ is $2\sum_{k=0}^{\lfloor \frac{\tilde{n}+1}{2} \rfloor}\binom{\tilde{n}+1}{2k+1}=2^{\tilde{n}+1}$.
When $\tilde{n}$ is odd, we verify $r_l\,Q_{\tilde{n}}(s_l)=0$ with $s_l=iz\sin{\left(\frac{l\pi}{\tilde{n}+1}\right)}$, $l=0,\ldots, \tilde{n}$. For $l\le \frac{\tilde{n}+1}{2}$, we have $r_l=\sqrt{z^2+s_l^2}=z\cos{\left(\frac{l\pi}{\tilde{n}+1}\right)}$ and 
\begin{align*}
r_l\,Q_{\tilde{n}}(s_l)
&=(s_l+r_l)^{\tilde{n}+1}-(r_l-s_l)^{\tilde{n}+1}\\
&=z^{\tilde{n}+1}\,\bigg[\left(\cos{\left(\frac{l\pi}{\tilde{n}+1}\right)}+i\sin{\left(\frac{l\pi}{\tilde{n}+1}\right)}\right)^{\tilde{n}+1}\\
&\quad -\left(\cos{\left(\frac{l\pi}{\tilde{n}+1}\right)}-i\sin{\left(\frac{l\pi}{\tilde{n}+1}\right)}\right)^{\tilde{n}+1}\bigg]\\
&=z^{\tilde{n}+1}\left(e^{il\pi}-e^{-il\pi}\right)\\
&=2\,iz^{\tilde{n}+1}\,\sin{(l\pi)}\\
&=0.
\end{align*}
For $l>\frac{\tilde{n}+1}{2}$, $r_l\,Q_{\tilde{n}}(s_l)=0$ can be proved in the same way. Moreover, we have $r_l=0$ if and only if $l=\frac{\tilde{n}+1}{2}$. Hence, $s_l,\, l\ne \frac{\tilde{n}+1}{2}$ are the $\tilde{n}$ roots of the polynomial $Q_{\tilde{n}}(s)$. We get the factorization
\[
Q_{\tilde{n}}(s)=2^{\tilde{n}+1}\prod_{\br{l=0,}{l\ne\frac{\tilde{n}+1}{2}}}^{\tilde{n}}\left(s-iz\sin{\left(\frac{l\pi}{\tilde{n}+1}\right)}\right).
\]
\item For the numerator $O_{\tilde{n}-1}(s)$,  since
\begin{align*}
O_{\tilde{n}-1}(s)&=\frac{(-1)^{\tilde{n}}(r-s)^{\tilde{n}}-(s+r)^{\tilde{n}}}{r}\\
&=-\frac{1}{r}\sum_{k=0}^{\tilde{n}}\binom{\tilde{n}}{k}\,s^{\tilde{n}-k}\,r^k\,(1-(-1)^k)\\
&=-2\sum_{k=0}^{\lfloor \frac{\tilde{n}}{2} \rfloor}\binom{\tilde{n}}{2k+1}\,s^{\tilde{n}-2k-1}\,(s^2+z^2)^k,
\end{align*}
we have that $O_{\tilde{n}-1}(s)$ is a polynomial of degree $\tilde{n}-1$ in $s$. The coefficient of $s^{\tilde{n}-1}$ is $-2\sum_{k=0}^{\lfloor \frac{\tilde{n}}{2} \rfloor}\binom{\tilde{n}}{2k+1}=-2^{\tilde{n}}$. \\
Next, we verify that $s_l=iz\cos{\left(\frac{l\pi}{\tilde{n}}\right)}=iz\sin{\left(\frac{\pi}{2}+\frac{l\pi}{\tilde{n}}\right)},\, l=0,\ldots,\tilde{n}-1$ are $\tilde{n}$ roots of $O_{\tilde{n}-1}(s_l)=0$ when $\tilde{n}$ is odd. We compute
\begin{align*}
r_l\,O_{\tilde{n}-1}(s_l)
&=-\left((s_l+r_l)^{\tilde{n}}+(r_l-s_l)^{\tilde{n}}\right)\\
&=-z^{\tilde{n}}\,\bigg[\left(\cos{\left(\frac{\pi}{2}+\frac{l\pi}{\tilde{n}}\right)}+i\sin{\left(\frac{\pi}{2}+\frac{l\pi}{\tilde{n}}\right)}\right)^{\tilde{n}}\\
&\quad +\left(\cos{\left(\frac{\pi}{2}+\frac{l\pi}{\tilde{n}}\right)}-i\sin{\left(\frac{\pi}{2}+\frac{l\pi}{\tilde{n}}\right)}\right)^{\tilde{n}}\bigg]\\
&=-z^{\tilde{n}}\left(e^{i\left(\frac{\pi}{2}\tilde{n}+l\pi\right)}+e^{-i\left(\frac{\pi}{2}\tilde{n}+l\pi\right)}\right)\\
&=-2z^{\tilde{n}}\,\cos{\left(\frac{\pi}{2}\tilde{n}+l\pi\right)}\\
&=0.
\end{align*}
Note that $r_l=0$ if and only if $l=0$. So $s_l,\, l=1,\ldots,\tilde{n}-1$ are the $n-1$ roots of the polynomial $O_{\tilde{n}-1}(s)$. Hence, we have 
\[
O_{\tilde{n}-1}(s)=-2^{\tilde{n}}\prod_{l=1}^{\tilde{n}-1}\left(s-iz\sin{\left(\frac{\pi}{2}+\frac{l\pi}{\tilde{n}}\right)}\right).
\]
\end{enumerate}
The case $\tilde{n}$ even is treated similarly. 
\end{proof}

\begin{remark}
It is seen that all the rational functions in the kernel are proper (the degree of the numerator polynomial is smaller than the degree of the  denominator polynomial). Therefore the function in the bivector part of $\mathcal{L} (K_{2}^c)$ for case $\tilde{n}$ odd
\[
\frac{\,\prod_{l=1}^{\tilde{n}-1}\left(s-iz\cos{\left(\frac{l\pi}{\tilde{n}}\right)}\right)}
{{\prod_{l=0}^{n-1}\left(s+iz\cos{\left(\frac{\theta+2\pi l}{n}\right)}\right)}}
:=X(s)
\]
has the partial fraction expansion
\[
\frac{c_0}{s+iz\cos{\left(\frac{\theta}{n}\right)}}
+\frac{c_1}{s+iz\cos{\left(\frac{\theta+2\pi}{n}\right)}}
+\cdots +\frac{c_{n-1}}{s+iz\cos{\left(\frac{\theta+2\pi (n-1)}{n}\right)}}
\]
where the coefficients $c_0, \ldots , c_{n-1}$ $($called the residues of $X(s)$$)$ are determined via the formula
\[
c_l=\left[\left(s+iz\cos{\left(\frac{\theta+2\pi l}{n}\right)}\right)X(s)\right]_{s=-iz\cos{\left(\frac{\theta+2\pi l}{n}\right)}}
\]
for $l=0, \ldots, n-1$.
Using linearity and formula \eqref{exponlapfor} for each term, we have
\begin{align*}
K_{2, biv}^c
&=\frac{i}{2^{\tilde{n}+1}}\mathcal{L}^{-1}\left(\sum_{l=0}^{n-1}\frac{c_l}{s+iz\cos{\left(\frac{\theta+2\pi l}{n}\right)}}\right)\\
&=\frac{i}{2^{\tilde{n}+1}}\sum_{l=0}^{n-1}c_l\,e^{-izt\cos{\left(\frac{\theta+2\pi l}{n}\right)}}, \quad l=0, \ldots, n-1.
\end{align*}
The scalar part of the kernel can be considered by analogy with the bivector case. This form of kernel can be compared with the results of \cite{DBDSE-2017}, given in Theorem \ref{DBDSEker2} by setting $t=1$.

\end{remark}
\subsection{The dimension $m>2$}\label{Generalevenker}
For dimensions $m>2$, we rewrite the kernel  in \eqref{p29} by introducing an auxiliary variable $t$ in the Bessel functions. For $1+c=\frac{1}{n}$, $n\in \mathbb{N}_0\backslash\ \{1\}$ odd and $n=2\tilde{n}+1$, $\tilde{n}\in\mathbb{N}_0$, the kernel of the deformed Fourier transform is given by
\begin{equation}\label{Newkernelsmeven}
\begin{aligned}
A_{\lambda} &= \sum_{k=0}^{+\infty} (k+\lambda)\, (-i)^{kn} {J}_{nk+n\lambda+\tilde{n}}(z\,t)\,C_k^{\lambda}(w)\\
 &\quad +i^{n-1}\,\sum_{k=0}^{+\infty} (k+\lambda)\, (-i)^{kn} {J}_{nk+n\lambda-\tilde{n}}(z\,t)\,C_k^{\lambda}(w);\\
B_{\lambda} &=  \sum_{k=0}^{+\infty} (-i)^{kn} {J}_{nk+n\lambda+\tilde{n}}(z\,t)\,C_k^{\lambda}(w)\\
&\quad -i^{n-1}\, \sum_{k=0}^{+\infty}(-i)^{kn} {J}_{nk+n\lambda-\tilde{n}}(z\,t)\,C_k^{\lambda}(w);\\
C_{\lambda} &=  (-i)^{n}\,\sum_{k=0}^{+\infty}(-i)^{kn} {J}_{nk+n\lambda+3\tilde{n}+1}(z\,t)\,C_{k}^{\lambda+1}(w)\\
&\quad +i\, \sum_{k=0}^{+\infty}(-i)^{kn}{J}_{nk+n\lambda+\tilde{n}+1}(z\,t)\,C_{k}^{\lambda+1}(w)
\end{aligned}
\end{equation}
where we substituted $k\mapsto k+1$ in $C_{\lambda}$.\\
To consider the Laplace transform of $A_{\lambda}$, we need the  expansion of the Poisson kernel in terms of Gegenbauer polynomials \eqref{Poistheo}.
By Theorem \ref{laptranguarantee}, Theorem \ref{Poistheo} and formula \eqref{LapBessfor}, we take the Laplace transform with respect to $t$ in \eqref{Newkernelsmeven} for ${\rm Re}\,s$ big enough. With $r=\sqrt{s^2+z^2}$, $u_R=\left(\frac{-iz}{s+r}\right)^n$, $\lambda = \frac{m-2}{2}>0$, $\uz = |x||y|$, $w = \frac{\langle x, \, y \rangle}{z},$ we obtain
\begin{equation}\label{2DimenkernelA}
\begin{aligned}
\mathcal{L}\left(\frac{1}{2\lambda}z^{-\frac{\mu-2}{2}}A_{\lambda}\right)
&=\frac{1}{2r}\left(\frac{1}{s+r}\right)^{n\lambda+\tilde{n}}\frac{1-u_R^2}{(1-2w u_R+u_R^2)^{\lambda+1}}\\
&\quad + \left(\frac{i}{z}\right)^{n-1}\frac{1}{2r}\left(\frac{1}{s+r}\right)^{n\lambda-\tilde{n}}\frac{1-u_R^2}{(1-2w u_R+u_R^2)^{\lambda+1}}.
\end{aligned}
\end{equation}
For $B_{\lambda}$ and $C_{\lambda}$, using the formula \eqref{GeneFuncGegen}, we have, for $\lambda >0$, 
\begin{equation}\label{2DimenkernelB}
\begin{aligned}
\mathcal{L}\left(\frac{1}{2}z^{-\frac{\mu-2}{2}}B_{\lambda}\right)
&=\frac{1}{2r}\left(\frac{1}{s+r}\right)^{n\lambda+\tilde{n}}\frac{1}{(1-2w u_R+u_R^2)^{\lambda}}\\
&\quad - \left(\frac{i}{z}\right)^{n-1}\frac{1}{2r}\left(\frac{1}{s+r}\right)^{n\lambda-\tilde{n}}\frac{1}{(1-2w u_R+u_R^2)^{\lambda}}
\end{aligned}
\end{equation}
and
\begin{equation}\label{2DimenkernelC}
\begin{aligned}
\mathcal{L}\left(z^{-\frac{\mu}{2}}C_{\lambda}\right)
&=(-i)^nz^{n-1}\frac{1}{r}\left(\frac{1}{s+r}\right)^{n\lambda+3\tilde{n}+1}\frac{1}{(1-2w u_R+u_R^2)^{\lambda+1}}\\
&\quad + \frac{i}{r}\left(\frac{1}{s+r}\right)^{n\lambda+\tilde{n}+1}\frac{1}{(1-2w u_R+u_R^2)^{\lambda+1}}.
\end{aligned}
\end{equation}
By relation \eqref{srrsz2}, the kernels can be further simplified as follows 
\begin{align*}
\mathcal{L}\left(\frac{1}{2\lambda}z^{-\frac{\mu-2}{2}}A_{\lambda}\right)
&=\frac{1}{2r}\left(\frac{1}{s+r}\right)^{\tilde{n}}\frac{(s+r)^n-\frac{(-iz)^{2n}}{(s+r)^n}}{((s+r)^n-2w \,(-iz)^n+\frac{(-iz)^{2n}}{(s+r)^n})^{\lambda+1}}\\
&\quad + \left(\frac{i}{z}\right)^{n-1}\frac{1}{2r}\left(\frac{1}{s+r}\right)^{-\tilde{n}}\frac{(s+r)^n-\frac{(-iz)^{2n}}{(s+r)^n}}{((s+r)^n-2w \,(-iz)^n+\frac{(-iz)^{2n}}{(s+r)^n})^{\lambda+1}}\\
&=\frac{1}{2r}\frac{\left({s+r}\right)^{-\tilde{n}}\,\left((s+r)^n-(-1)^n\,(r-s)^n\right)}{((s+r)^n-2w \,(-iz)^n+(-1)^n\,(r-s)^n)^{\lambda+1}}\\
&\quad + \left(\frac{i}{z}\right)^{n-1}\frac{1}{2r}
\frac{\left(s+r\right)^{\tilde{n}}\left((s+r)^n-(-1)^n\,(r-s)^n\right)}{((s+r)^n-2w \,(-iz)^n+(-1)^n\,(r-s)^n)^{\lambda+1}}\\
&=\frac{1}{2r}\frac{\left((s+r)^{-\tilde{n}}+(-1)^{\tilde{n}}(r-s)^{-\tilde{n}}\right)\left((s+r)^{n}-(-1)^n(r-s)^{n}\right)}
{{\left((s+r)^n-2w (-iz)^n+(-1)^n(r-s)^n\right)^{\lambda +1}}}.
\end{align*}
The corresponding results for $B_{\lambda}$ and $C_{\lambda}$ can be derived immediately by the following expressions
\begin{align*}
&\quad \mathcal{L}\left(\frac{1}{2}z^{-\frac{\mu-2}{2}}B_{\lambda}\right)\\
&=\frac{1}{2r}\left(\frac{1}{s+r}\right)^{\tilde{n}}\frac{1}{((s+r)^n-2w \,(-iz)^n+(-1)^n(r-s)^n)^{\lambda}}\\
&\quad - \left(\frac{i}{z}\right)^{n-1}\frac{1}{2r}\left(\frac{1}{s+r}\right)^{-\tilde{n}}\frac{1}{((s+r)^n-2w \,(-iz)^n+(-1)^n(r-s)^n)^{\lambda}}
\end{align*}
and
\begin{align*}
\mathcal{L}\left(z^{-\frac{\mu}{2}}C_{\lambda}\right)
&=(-i)^{n}z^{n-1}\frac{1}{r}\left(\frac{1}{s+r}\right)^{\tilde{n}}\frac{1}{((s+r)^n-2w \,(-iz)^n+(-1)^n(r-s)^n)^{\lambda +1}}\\
&\quad + \frac{i}{r}\left(\frac{1}{s+r}\right)^{-\tilde{n}}\frac{1}{((s+r)^n-2w \,(-iz)^n+(-1)^n(r-s)^n)^{\lambda +1}}.
\end{align*}
Hence, we conclude the following theorem.
\begin{theorem}\label{Laplacekernels}
If the dimension $m>2\,(\lambda >0)$ and $1+c=\frac{1}{n}$, $n\in \mathbb{N}_0\backslash\ \{1\}$ odd, $n=2\tilde{n}+1$, $\tilde{n}\in\mathbb{N}_0$, then the kernel of the radially deformed Fourier transform $\cF_{D} =   e^{i\frac{\pi}{2}\left( \frac{1}{2}+\frac{m-1}{2(1+c)}\right)}e^{\frac{-i\pi}{4(1+c)^2}\left({\bf D}^2 - (1+c)^2\ux^2\right)}$ in the Laplace domain takes the form
\begin{align*}
\mathcal{L}\left(K_{m}^c\right) = \mathcal{L}\left(\frac{1}{2\lambda}z^{-\frac{\mu -2}{2}} A_{\lambda}\right)+ \mathcal{L}\left(\frac{1}{2}z^{-\frac{\mu -2}{2}}B_{\lambda}\right)-(\ux \wedge \uy)\,\mathcal{L}\left(z^{-\frac{\mu }{2}} \ C_{\lambda}\right)
\end{align*}
with
\begin{align*} 
\mathcal{L}\left(\frac{1}{2\lambda}z^{-\frac{\mu-2}{2}}A_{\lambda}\right)
&=\frac{1}{2r}\frac{\left((s+r)^{n}-(-1)^n(r-s)^{n}\right)\left((s+r)^{-\tilde{n}}+(-1)^{\tilde{n}}(r-s)^{-\tilde{n}}\right)}
{{\left((s+r)^n-2w (-iz)^n+(-1)^n(r-s)^n\right)^{\lambda +1}}};\\
\mathcal{L}\left(\frac{1}{2}z^{-\frac{\mu-2}{2}}B_{\lambda}\right)&=\frac{1}{2r}\frac{(s+r)^{-\tilde{n}}-(-1)^{\tilde{n}}(r-s)^{-\tilde{n}}}
{{\left((s+r)^n-2w (-iz)^n+(-1)^n(r-s)^n\right)^{\lambda}}};\\
\mathcal{L}\left(z^{-\frac{\mu}{2}}C_{\lambda}\right)&=\frac{i}{r}\frac{(s+r)^{\tilde{n}}-(-1)^{\tilde{n}}(r-s)^{\tilde{n}}}
{{\left((s+r)^n-2w(-iz)^n+(-1)^n(r-s)^n\right)^{\lambda +1}}}
\end{align*}
where $r=\sqrt{s^2+z^2},\, z=|x||y|,\, w = \frac{\langle x, \, y \rangle}{z}$.
\end{theorem}
\begin{remark}\label{ADDevenker}
When $\lambda>0$, we can rewrite the expressions into the following forms
\begin{align*} 
\mathcal{L}\left(\frac{1}{2\lambda}z^{-\frac{\mu-2}{2}}A_{\lambda}\right)
&=\frac{1}{2rz^{2\tilde{n}}}\frac{(-1)^{\tilde{n}}(s+r)^{3\tilde{n}+1}+(r-s)^{3\tilde{n}+1}
}{{\left((s+r)^n-2w (-iz)^n+(-1)^n(r-s)^n\right)^{\lambda +1}}}\\
&\quad + \frac{1}{2r}\frac{(s+r)^{\tilde{n}+1}+(-1)^{\tilde{n}}(r-s)^{\tilde{n}+1}
}{{\left((s+r)^n-2w (-i z)^n+(-1)^n(r-s)^n\right)^{\lambda +1}}};\\
\mathcal{L}\left(\frac{1}{2}z^{-\frac{\mu-2}{2}}B_{\lambda}\right)&=\frac{1}{2rz^{2\tilde{n}}}\frac{(r-s)^{\tilde{n}}-(-1)^{\tilde{n}}(s+r)^{\tilde{n}}}
{{\left((s+r)^n-2w (-iz)^n+(s-r)^n\right)^{\lambda}}};\\
\mathcal{L}\left(z^{-\frac{\mu}{2}}C_{\lambda}\right)&=\frac{i}{r}\frac{(s+r)^{\tilde{n}}-(-1)^{\tilde{n}}(r-s)^{\tilde{n}}}
{{\left((s+r)^n-2w (-iz)^n+(s-r)^n\right)^{\lambda +1}}}
\end{align*}
by the relation
\begin{align*}
(s+r)^{-\tilde{n}}-(-1)^{\tilde{n}}(r-s)^{-\tilde{n}}
&=\frac{\left((r-s)^{\tilde{n}}-(-1)^{\tilde{n}}(s+r)^{\tilde{n}}\right)}{z^{2\tilde{n}}}.
\end{align*}
\end{remark}
If we consider the Laplace transform of $B_{\lambda}$ in Theorem \ref{Laplacekernels} in the following forms
\begin{align*}
\mathcal{L}\left(\frac{1}{2}z^{-\frac{\mu-2}{2}}B_{\lambda}\right)
&=\frac{1}{2r}\frac{(s+r)^{\tilde{n}+1}+(-1)^{\tilde{n}}(r-s)^{\tilde{n}+1}
}{\left((s+r)^n-2w (-iz)^n+(-1)^{n}(r-s)^n\right)^{\lambda +1}}\\
&\quad +\frac{i\,w z}{r}\frac{(-1)^{\tilde{n}}(r-s)^{\tilde{n}}-(s+r)^{\tilde{n}}}
{\left((s+r)^n-2w (-iz)^n+(-1)^{n}(r-s)^n\right)^{\lambda +1}}\\
&\quad -\frac{z^{-2\tilde{n}}}{2r}\frac{(-1)^{\tilde{n}}(s+r)^{3\tilde{n}+1}+(r-s)^{3\tilde{n}+1}
}{\left((s+r)^n-2w (-iz)^n+(-1)^{n}(r-s)^n\right)^{\lambda +1}},
\end{align*}
then summing the results of $A_{\lambda}$ in Remark \ref{ADDevenker} together as the scalar part of the kernel, we arrive at the following result:
\begin{theorem}
\label{LaplBig2}
If the dimension $m>2\,(\lambda >0)$ and $1+c=\frac{1}{n}$, $n\in \mathbb{N}_0\backslash\ \{1\}$ odd, $n=2\tilde{n}+1$, $\tilde{n}\in\mathbb{N}_0$, the kernel of the radially deformed Fourier transform $\cF_{D} =   e^{i\frac{\pi}{2}\left( \frac{1}{2}+\frac{m-1}{2(1+c)}\right)}e^{\frac{-i\pi}{4(1+c)^2}\left({\bf D}^2 - (1+c)^2\ux^2\right)}$ is given by
\[
K_{m}^c =K_{m,scal}^c+ (\ux \wedge \uy)\,K_{m,biv}^c
\]
with
\begin{align*}
K_{m,scal}^c&=\frac{1}{r}\frac{(s+r)^{\tilde{n}+1}+(-1)^{\tilde{n}}(r-s)^{\tilde{n}+1}}
{\left((s+r)^n-2w (-iz)^n+(-1)^{n}(r-s)^n\right)^{\lambda +1}}\\
&\quad +\frac{w z\,i}{r}\frac{(-1)^{\tilde{n}}(r-s)^{\tilde{n}}-(s+r)^{\tilde{n}}}
{\left((s+r)^n-2w (-iz)^n+(-1)^{n}(r-s)^n\right)^{\lambda +1}}
\end{align*}
and 
\begin{align*}
K_{m,biv}^c=\frac{i}{r}\frac{(s+r)^{\tilde{n}}-(-1)^{\tilde{n}}(r-s)^{\tilde{n}}}
{\left((s+r)^n-2w (-iz)^n+(-1)^n(r-s)^n\right)^{\lambda +1}}
\end{align*}
where $r=\sqrt{s^2+z^2},\, z=|x||y|,\, w = \frac{\langle x, \, y \rangle}{z}$.
\end{theorem}
\begin{remark}
If we take the limit $\lambda =0$ in Theorem \ref{LaplBig2}, 
we can recover the expressions for dimension 2 in Theorem \ref{kerneldim2}.
\end{remark}
\begin{remark}
When $m$ is even, the factorization for the polynomials in the above theorem can be derived immediately from Theorem \ref{Ker2coroll}.
\end{remark}
\section{Explicit expressions of the kernel}\label{MINLeffexPofkernel}
 In this section, we give the integral expressions for the kernel in all dimensions. These results will be given in terms of the Mittag-Leffler functions in Definition \ref{defML} and Definition \ref{defPML} by adapting the method introduced in \cite{DDBP-2018}. 
\subsection{The case of $\lambda=0$}
We consider the kernel of the deformed Fourier transform in the forms of \eqref{scakersr} and \eqref{bivkersr} in dimension 2, i.e. $\lambda =0$, for special values of the deformation parameter $c$.

Let us first calculate the 4 different terms separately given in \eqref{scakersr}
\begin{align}\label{I1234}
\mathcal{L}\left(K_{2,scal}^c\right)
:=I_1+I_2+I_3+I_4
\end{align}
with
\begin{align*}
I_1&=\frac{1}{r}\left(\frac{1}{s+r}\right)^{\tilde{n}}\frac{1}{1-2w\, u_R+u_R^2},\\
I_2&=-\frac{1}{r}\left(\frac{1}{s+r}\right)^{\tilde{n}}\frac{w\, u_R}{1-2w\, u_R+u_R^2},\\
I_3&=(-1)^{\tilde{n}}z^{-2\tilde{n}}\frac{1}{r}\left(\frac{1}{s+r}\right)^{-\tilde{n}}\frac{w\, u_R}{1-2w\, u_R+u_R^2},\\
I_4&=-(-1)^{\tilde{n}}z^{-2\tilde{n}}\frac{1}{r}\left(\frac{1}{s+r}\right)^{-\tilde{n}}\frac{u_R^2}{1-2w\, u_R+u_R^2},
\end{align*}
where $u_R=\left(\frac{-iz}{s+r}\right)^n$, $r=\sqrt{s^2 +z^2}$ and $w =\cos{\theta}$.  Denote $b_{\pm}=i^n\,e^{\pm i\theta}z^n$. Then the common fraction can be equivalently written as
\begin{equation}\label{uRrewrite}
\begin{aligned}
\frac{1}{1-2w\, u_R+u_R^2}
&=\frac{1}{u_R-e^{i\theta}}\frac{1}{u_R-e^{-i\theta}}\\
&=\frac{i^n\,z^{n}}{(r-s)^n-b_+}\frac{i^n\,z^{n}}{(r-s)^n-b_-}\\
&=(-1)^n\,z^{2n}\,\frac{1}{(r-s)^n-b_+}\frac{1}{(r-s)^n-b_-},
\end{aligned}
\end{equation}
where we used \eqref{srrsz2} to rewrite:
\begin{align*}
u_R=\left(\frac{-iz}{s+r}\right)^n=(-iz)^n\,\frac{(r-s)^n}{z^{2n}}=(-i)^n\,\frac{(r-s)^n}{z^n}.
\end{align*}
Take into account \eqref{MLinvers},  we rewrite $I_1$ up to $I_4$ as 
\begin{align*}
I_1&=-\frac{1}{r}\,z^{n+1}\,(r-s)^{\tilde{n}}\,\frac{1}{(r-s)^n-b_+}\frac{1}{(r-s)^n-b_-};\\
I_2&=-\frac{1}{r}\,\omega\,i^n\, z\,(r-s)^{n+\tilde{n}}\,\frac{1}{(r-s)^n-b_+}\frac{1}{(r-s)^n-b_-};\\
I_3&=\frac{i}{r}\,\omega\, z^n\,(r-s)^{\tilde{n}+1}\,\frac{1}{(r-s)^n-b_+}\frac{1}{(r-s)^n-b_-};\\
I_4&=-(-1)^{\tilde{n}}\frac{1}{r}\,(r-s)^{3\tilde{n}+2}\,\frac{1}{(r-s)^n-b_+}\frac{1}{(r-s)^n-b_-}
\end{align*}
in the form $F(r-s)=F\left(\sqrt{s^2+z^2}-s\right)$ by using relation \eqref{srrsz2} again.\\
Next, we consider the inverse Laplace transform of $\frac{1}{s^n-b_+}\frac{1}{s^n-b_-}$. 
By Remark \ref{remark2variML} and Formula \eqref{LaPPrabhaMLe}, we derive the Laplace transform of the Mittag-Leffler function \eqref{parbmitlef}
\begin{equation*}
\mathcal{L}\left(t^{\beta -1}E_{\alpha, \beta} (b\,t^{\alpha})\right)=\frac{1}{s^{\beta}}\frac{1}{1-bs^{-\alpha}}
\end{equation*}
where ${\rm Re}\,\alpha >0, {\rm Re}\,\beta>0, {\rm Re}\,s>0$ and $s>|b|^{1/({\rm Re}\,\alpha)}$.
Then we obtain the inverse transform
\begin{equation*}
\mathcal{L}^{-1}\left(\frac{1}{s^n-b}\right)=t^{n-1}\,E_{n,n}\left(b\,t^n\right).
\end{equation*}
Setting $F(s)=\frac{1}{s^n-b_+}$, $G(s)=\frac{1}{s^n-b_-}$ and using formula \eqref{convoluLap}, the inverse Laplace transform is given by
\begin{align*}
h_{1}(t)&:=\mathcal{L}^{-1}\left(\frac{1}{s^n-b_+}\frac{1}{s^n-b_-}\right)\\
&=\int_0^t\zeta ^{n -1}E_{n, n}\left(b_+\,\zeta ^n\right)(t-\zeta)^{n-1}E_{n,n}(b_-(t-\zeta)^n)\,\mathrm{d}\zeta.
\end{align*}
Therefore, adding these 4 terms and using \eqref{MLinvers}, we arrive at the following result
\begin{align*}
K_{2,scal}^c
&=-z^{3\tilde{n}+2}\int_0^1(1+2\tau)^{-\tilde{n}/2}\,J_{\tilde{n}}\left(z\sqrt{1+2\tau}\right)\,h_1(\tau)\,\mathrm{d}\tau\\
&\quad -iw\,(-1)^{\tilde{n}}\,z^{3\tilde{n}+2}\int_0^1(1+2\tau)^{-(3\tilde{n}+1)/2}\,J_{3\tilde{n}+1}\left(z\sqrt{1+2\tau}\right)\,h_1(\tau)\,\mathrm{d}\tau\\
&\quad +iw\,z^{3\tilde{n}+2}\int_0^1(1+2\tau)^{-(\tilde{n}+1)/2}\,J_{\tilde{n}+1}\left(z\sqrt{1+2\tau}\right)\,h_1(\tau)\,\mathrm{d}\tau\\
&\quad -(-1)^{\tilde{n}}\,z^{3\tilde{n}+2}\int_0^1(1+2\tau)^{-(3\tilde{n}+2)/2}\,J_{3\tilde{n}+2}\left(z\sqrt{1+2\tau}\right)\,h_1(\tau)\,\mathrm{d}\tau.
\end{align*}
Similarly, we deduce the Laplace transform of the bivector part \eqref{bivkersr}
\begin{align*}
\mathcal{L}\left(K_{2,biv}^c\right)
&=-i^n\,\frac{1}{r}(r-s)^{n+\tilde{n}}\frac{1}{(r-s)^n-b_+}\frac{1}{(r-s)^n-b_-}\\
&\quad +iz^{2\tilde{n}}\frac{1}{r}(r-s)^{n-\tilde{n}}\frac{1}{(r-s)^n-b_+}\frac{1}{(r-s)^n-b_-}.
\end{align*}
It then follows that 
\begin{align*}
K_{2,biv}^c
&=-i^n\,z^{3\tilde{n}+1}\int_0^1(1+2\tau)^{-(3\tilde{n}+1)/2}\,J_{3\tilde{n}+1}\left(z\sqrt{1+2\tau}\right)\,h_1(\tau)\,\mathrm{d}\tau\\
&\quad +iz^{3\tilde{n}+1}\int_0^1(1+2\tau)^{-(\tilde{n}+1)/2}\,J_{\tilde{n}+1}\left(z\sqrt{1+2\tau}\right)\,h_1(\tau)\,\mathrm{d}\tau.
\end{align*}
These results immediately lead to the following theorem.
\begin{theorem}\label{InteGralKerDim2}
Let $b_{\pm}:=i^n\,e^{\pm i\theta}z^n$ and
\begin{align*}
h_{1}(t)
&:=\int_0^t\zeta ^{n -1}E_{n, n}\left(b_+\,\zeta ^n\right)(t-\zeta)^{n-1}E_{n,n}(b_-(t-\zeta)^n)\,\mathrm{d}\zeta.
\end{align*}
Then for $m=2 \,(\lambda =0)$, $1+c=\frac{1}{n}, \, n \in \mathbb{N}_0\backslash\ \{1\}$ with $n$ odd and $n=2\tilde{n}+1$, $\tilde{n}\in\mathbb{N}_0$, the kernel of the deformed Fourier transform $\cF_{D} =   e^{i\frac{\pi}{2}\left( \frac{1}{2}+\frac{m-1}{2(1+c)}\right)}e^{\frac{-i\pi}{4(1+c)^2}\left({\bf D}^2 - (1+c)^2\ux^2\right)}$ takes the form
\begin{align*}
K_2^c=K_{2,scal}^c+ (\ux \wedge \uy)\,K_{2,biv}^c
\end{align*}
with
\begin{align*}
K_{2,scal}^c
&=-z^{3\tilde{n}+2}\int_0^1(1+2\tau)^{-\tilde{n}/2}\,J_{\tilde{n}}\left(z\sqrt{1+2\tau}\right)\,h_1(\tau)\,\mathrm{d}\tau\\
&\quad -iw\,(-1)^{\tilde{n}}\,z^{3\tilde{n}+2}\int_0^1(1+2\tau)^{-(3\tilde{n}+1)/2}\,J_{3\tilde{n}+1}\left(z\sqrt{1+2\tau}\right)\,h_1(\tau)\,\mathrm{d}\tau\\
&\quad +iw\,z^{3\tilde{n}+2}\int_0^1(1+2\tau)^{-(\tilde{n}+1)/2}\,J_{\tilde{n}+1}\left(z\sqrt{1+2\tau}\right)\,h_1(\tau)\,\mathrm{d}\tau\\
&\quad -(-1)^{\tilde{n}}\,z^{3\tilde{n}+2}\int_0^1(1+2\tau)^{-(3\tilde{n}+2)/2}\,J_{3\tilde{n}+2}\left(z\sqrt{1+2\tau}\right)\,h_1(\tau)\,\mathrm{d}\tau
\end{align*}
and
\begin{align*}
K_{2,biv}^c
&=-i^n\,z^{3\tilde{n}+1}\int_0^1(1+2\tau)^{-(3\tilde{n}+1)/2}\,J_{3\tilde{n}+1}\left(z\sqrt{1+2\tau}\right)\,h_1(\tau)\,\mathrm{d}\tau\\
&\quad +iz^{3\tilde{n}+1}\int_0^1(1+2\tau)^{-(\tilde{n}+1)/2}\,J_{\tilde{n}+1}\left(z\sqrt{1+2\tau}\right)\,h_1(\tau)\,\mathrm{d}\tau
\end{align*}
where $w =\cos{\theta}=\frac{\langle x, \, y \rangle}{z}$, $z=|x||y|$.
\end{theorem}
\subsection{The case of $\lambda >0$}
In this subsection, we begin by considering the Laplace transform of the kernel given in (\ref{2DimenkernelA}-\ref{2DimenkernelC}) in Section \ref{Generalevenker}. Subsequently, we give the explicit expressions of kernel in terms of the Prabhakar generalized Mittag-Leffler functions \eqref{PRABmittleff} in dimension $m>2$. 

We start with the expression of $\mathcal{L}\left(\frac{1}{2\lambda}z^{-\frac{\mu-2}{2}}A_{\lambda}\right)$ shown in \eqref{2DimenkernelA}. It can be spilt in 4 different equations that we calculate separately:
\begin{align*}
\mathcal{L}\left(\frac{1}{2\lambda}z^{-\frac{\mu-2}{2}}A_{\lambda}\right)
&=\frac{1}{2r}\left(\frac{r-s}{z^2}\right)^{n\lambda+\tilde{n}}\frac{1}{(1-2w\, u_R+u_R^2)^{\lambda+1}}\\
&\quad -\frac{1}{2r}\left(\frac{r-s}{z^2}\right)^{n\lambda+\tilde{n}}\frac{u_R^2}{(1-2w\, u_R+u_R^2)^{\lambda+1}}\\
&\quad +\left(\frac{i}{z}\right)^{n-1}\frac{1}{2r}\left(\frac{r-s}{z^2}\right)^{n\lambda-\tilde{n}}\frac{1}{(1-2w\, u_R+u_R^2)^{\lambda+1}}\\
&\quad -\left(\frac{i}{z}\right)^{n-1}\frac{1}{2r}\left(\frac{r-s}{z^2}\right)^{n\lambda-\tilde{n}}\frac{u_R^2}{(1-2w\, u_R+u_R^2)^{\lambda+1}}\\
&:=I_1'+I_2'+I_3'+I_4'.
\end{align*}
As we have established in the previous subsection, 
\begin{align*}
\frac{1}{(1-2w\, u_R+u_R^2)^{\lambda +1}}
=\frac{(-1)^{ n(\lambda +1)}\,z^{2n(\lambda +1)}}{\left((r-s)^n-b_+\right)^{\lambda +1}}\frac{1}{\left((r-s)^n-b_-\right)^{\lambda +1}}
\end{align*}
where $b_{\pm}=e^{\pm i\theta}i^nz^n$, we subsequently obtain
\begin{align*}
I_1'=\frac{1}{2r}z^{n+1}(-1)^{ n(\lambda +1)}\frac{(r-s)^{n\lambda +\tilde{n}}}{\left((r-s)^n-b_+\right)^{\lambda +1}}\frac{1}{\left((r-s)^n-b_-\right)^{\lambda +1}}
\end{align*}
using \eqref{srrsz2}. Similarly for $I_2'$ 
\begin{align*}
I_2'=\frac{1}{2r}z^{1-n}(-1)^{ n\lambda+1}\frac{(r-s)^{n\lambda +2n+\tilde{n}}}{\left((r-s)^n-b_+\right)^{\lambda +1}}\frac{1}{\left((r-s)^n-b_-\right)^{\lambda +1}}
\end{align*}
and also
\begin{align*}
I_3'&=i^{n-1}\frac{1}{2r}z^{2n}(-1)^{ n(\lambda +1)}\frac{(r-s)^{n\lambda -\tilde{n}}}{\left((r-s)^n-b_+\right)^{\lambda +1}}\frac{1}{\left((r-s)^n-b_-\right)^{\lambda +1}};\\
I_4'&=i^{n-1}\frac{1}{2r}(-1)^{ n\lambda +1}\frac{(r-s)^{n\lambda +2n-\tilde{n}}}{\left((r-s)^n-b_+\right)^{\lambda +1}}\frac{1}{\left((r-s)^n-b_-\right)^{\lambda +1}}.
\end{align*}
Next we deduce the inverse Laplace transform by means of formulas \eqref{LaPPrabhaMLe} and \eqref{MLinvers}. Using the convolution formula \eqref{convoluLap} again, we have 
\begin{align*}
h_{\lambda +1}(t)&:=\mathcal{L}^{-1}\left(\frac{1}{(s^n-b_+)^{\lambda +1}}\frac{1}{(s^n-b_-)^{\lambda +1}}\right)\\
&=\int_0^t\zeta^{n(\lambda +1) -1}E^{\lambda +1}_{n, n(\lambda +1)}\left(b_+\,\zeta^n\right)(t-\zeta)^{n(\lambda +1)-1}E^{\lambda +1}_{n,n(\lambda +1)}\,(b_-(t-\zeta)^n)\,\mathrm{d}\tau.
\end{align*}
Now summing the inverse Laplace transform of $I_1'+I_2'+I_3'+I_4'$, we conclude the results as below:
\begin{align*}
&\quad\mathcal{L}^{-1}\left(\frac{1}{2\lambda}z^{-\frac{\mu-2}{2}}A_{\lambda}\right)\\
&= c_1\,z^{n\lambda +3\tilde{n} +2}\,\bigg(\int_0^1(1+2\tau)^{-(n\lambda+\tilde{n})/2}J_{n\lambda +\tilde{n}}\left(z\sqrt{1+2\tau}\right)h_{\lambda +1}(\tau)\,\mathrm{d}\tau\\
&\quad +\int_0^1(1+2\tau)^{-(n\lambda+2n+\tilde{n})/2}J_{n\lambda +2n+\tilde{n}}\left(z\sqrt{1+2\tau}\right)h_{\lambda +1}(\tau)\,\mathrm{d}\tau\\
&\quad +i^{n-1}\,\int_0^1(1+2\tau)^{-(n\lambda-\tilde{n})/2}J_{n\lambda -\tilde{n}}\left(z\sqrt{1+2\tau}\right)h_{\lambda +1}(\tau)\,\mathrm{d}\tau\\
&\quad +i^{n-1}\,\int_0^1(1+2\tau)^{-(n\lambda+2n-\tilde{n})/2}J_{n\lambda +2n-\tilde{n}}\left(z\sqrt{1+2\tau}\right)h_{\lambda +1}(\tau)\mathrm{d}\tau\bigg),
\end{align*}
with $c_1=\frac{1}{2}(-1)^{n(\lambda +1)}$. \\
The expression for $C_{\lambda }$ in \eqref{2DimenkernelC} yields
\begin{align*}
\mathcal{L}\left(z^{-\frac{\mu}{2}}C_{\lambda}\right)
&=(-i)^n(-1)^{ n(\lambda +1)}\frac{1}{r}\frac{(r-s)^{n\lambda +3\tilde{n}+1}}{\left((r-s)^n-b_+\right)^{\lambda +1}}\frac{1}{\left((r-s)^n-b_-\right)^{\lambda +1}}\\
&\quad +i(-1)^{ n(\lambda +1)}z^{2\tilde{n}}\frac{1}{r}\frac{(r-s)^{n\lambda +\tilde{n}+1}}{\left((r-s)^n-b_+\right)^{\lambda +1}}\frac{1}{\left((r-s)^n-b_-\right)^{\lambda +1}}.
\end{align*}
It follows readily that
\begin{align*}
&\quad\mathcal{L}^{-1}\left(z^{-\frac{\mu}{2}}C_{\lambda}\right)\\
&=c_3\,z^{n\lambda+3\tilde{n}+1}\bigg(i\int_0^1(1+2\tau)^{-(n\lambda+\tilde{n}+1)/2}J_{n\lambda +\tilde{n}+1}\left(z\sqrt{1+2\tau}\right)h_{\lambda +1}(\tau)\,\mathrm{d}\tau\\
&\quad +(-i)^n\int_0^1(1+2\tau)^{-(n\lambda+3\tilde{n}+1)/2}J_{n\lambda +3\tilde{n}+1}\left(z\sqrt{1+2\tau}\right)h_{\lambda +1}(\tau)\,\mathrm{d}\tau\bigg)
\end{align*}
with $c_3=(-1)^{ n(\lambda +1)}$.
Similarly, we have by \eqref{convoluLap},
\begin{align*}
h_{\lambda}(t)&:=\mathcal{L}^{-1}\left(\frac{1}{\left(s^n-b_+\right)^{\lambda }}\frac{1}{\left(s^n-b_-\right)^{\lambda }}\right)\\
&=\int_0^t\zeta^{n\lambda -1}E^{\lambda }_{n, n\lambda }\left(b_+\,\zeta^n\right)
(t-\zeta)^{n\lambda-1}E^{\lambda}_{n,n\lambda }(b_-(t-\zeta)^n)\,\mathrm{d}\zeta.
\end{align*}
This allows us to compute the inverse Laplace transform of $B_{\lambda}$ obtained in \eqref{2DimenkernelB}
\begin{align*}
&\quad\mathcal{L}^{-1}\left(\frac{1}{2}z^{-\frac{\mu-2}{2}}B_{\lambda}\right)\\
&=\frac{1}{2}(-1)^{ n\lambda}z^{-2\tilde{n}}\,\mathcal{L}^{-1}\left(\frac{1}{r}\frac{(r-s)^{n\lambda +\tilde{n}}}{\left((r-s)^n-b_+\right)^{\lambda}}\frac{1}{\left((r-s)^n-b_-\right)^{\lambda}}\right)\\
&\quad -\frac{1}{2}i^{n-1}(-1)^{ n\lambda}\,\mathcal{L}^{-1}\left(\frac{1}{r}\frac{(r-s)^{n\lambda -\tilde{n}}}{\left((r-s)^n-b_+\right)^{\lambda}}\frac{1}{\left((r-s)^n-b_-\right)^{\lambda}}\right)\\
&=c_2\,z^{n\lambda-\tilde{n}}\int_0^1(1+2\tau)^{-(n\lambda+\tilde{n})/2}J_{n\lambda +\tilde{n}}\left(z\sqrt{1+2\tau}\right)h_{\lambda}(\tau)\,\mathrm{d}\tau\\
&\quad - i^{n-1}c_2\,z^{n\lambda-\tilde{n}}\int_0^1(1+2\tau)^{-(n\lambda-\tilde{n})/2}J_{n\lambda -\tilde{n}}\left(z\sqrt{1+2\tau}\right)h_{\lambda}(\tau)\,\mathrm{d}\tau,
\end{align*}
with $c_2=\frac{1}{2}(-1)^{ n\lambda}$.\\
Collecting all results then gives the desired theorem.
\begin{theorem}
Let $b_{\pm}:=e^{\pm i\theta}i^{ n}z^n$ and 
\begin{align*}
h_{\lambda}(t)
&:=\int_0^t\zeta^{n\lambda -1}E^{\lambda }_{n, n\lambda }\left(b_+\,\zeta^n\right)(t-\zeta)^{n\lambda-1}E^{\lambda}_{n,n\lambda }(b_-(t-\zeta)^n)\,\mathrm{d}\zeta.
\end{align*}
Then for all the dimensions $m>2\,(\lambda >0)$ and $1+c=\frac{1}{n}$, $n\in \mathbb{N}_0\backslash\ \{1\}$ odd, $n=2\tilde{n}+1$, $\tilde{n}\in\mathbb{N}_0$, the kernel of the radially deformed Fourier transform $\cF_{D} =   e^{i\frac{\pi}{2}\left( \frac{1}{2}+\frac{m-1}{2(1+c)}\right)}e^{\frac{-i\pi}{4(1+c)^2}\left({\bf D}^2 - (1+c)^2\ux^2\right)}$ is given by
\[
K_{m}^c = \frac{1}{2\lambda}z^{-\frac{\mu -2}{2}} A_{\lambda}+ \frac{1}{2}z^{-\frac{\mu -2}{2}}B_{\lambda}-z^{-\frac{\mu }{2}}\left(\ux \wedge \uy\right) \ C_{\lambda}
\]
with 
\begin{align*}
\frac{1}{2\lambda}z^{-\frac{\mu -2}{2}} A_{\lambda}
&= c_1\,z^{n\lambda +3\tilde{n} +2}\bigg(\int_0^1(1+2\tau)^{-(n\lambda+\tilde{n})/2}J_{n\lambda +\tilde{n}}\left(z\sqrt{1+2\tau}\right)h_{\lambda +1}(\tau)\,\mathrm{d}\tau\\
&\quad +\int_0^1(1+2\tau)^{-(n\lambda+2n+\tilde{n})/2}J_{n\lambda +2n+\tilde{n}}\left(z\sqrt{1+2\tau}\right)h_{\lambda +1}(\tau)\,\mathrm{d}\tau\\
&\quad +(-1)^{\tilde{n}}\int_0^1(1+2\tau)^{-(n\lambda-\tilde{n})/2}J_{n\lambda -\tilde{n}}\left(z\sqrt{1+2\tau}\right)h_{\lambda +1}(\tau)\,\mathrm{d}\tau\\
&\quad +(-1)^{\tilde{n}}\int_0^1(1+2\tau)^{-(n\lambda+2n-\tilde{n})/2}J_{n\lambda +2n-\tilde{n}}\left(z\sqrt{1+2\tau}\right)h_{\lambda +1}(\tau)\,\mathrm{d}\tau\bigg);\\
\frac{1}{2}z^{-\frac{\mu -2}{2}}B_{\lambda}
&=c_2\,z^{n\lambda-\tilde{n}}\bigg(\int_0^1(1+2\tau)^{-(n\lambda+\tilde{n})/2}J_{n\lambda +\tilde{n}}\left(z\sqrt{1+2\tau}\right)h_{\lambda}(\tau)\,\mathrm{d}\tau\\
&\quad - (-1)^{\tilde{n}}\int_0^1(1+2\tau)^{-(n\lambda-\tilde{n})/2}J_{n\lambda -\tilde{n}}\left(z\sqrt{1+2\tau}\right)h_{\lambda}(\tau)\,\mathrm{d}\tau\bigg);
\end{align*}
\begin{align*}
z^{-\frac{\mu}{2}}C_{\lambda}
&=c_3z^{n\lambda+3\tilde{n}+1}\bigg(i\int_0^1(1+2\tau)^{-(n\lambda+\tilde{n}+1)/2}J_{n\lambda +\tilde{n}+1}\left(z\sqrt{1+2\tau}\right)h_{\lambda +1}(\tau)\,\mathrm{d}\tau\\
&\quad +(-i)^n\int_0^1(1+2\tau)^{-(n\lambda+3\tilde{n}+1)/2}J_{n\lambda +3\tilde{n}+1}\left(z\sqrt{1+2\tau}\right)h_{\lambda +1}(\tau)\,\mathrm{d}\tau\bigg)
\end{align*}
where $w =\cos{\theta}=\frac{\langle x, \, y \rangle}{z}$, $z=|x||y|$, $c_1=\frac{1}{2}(-1)^{n(\lambda +1)}$, $c_2=\frac{1}{2}(-1)^{n\lambda}$ and $c_3=(-1)^{n(\lambda +1)}$.
\end{theorem}
It is now possible to present this result in a more compact form. Indeed, combining the equations in \eqref{2DimenkernelA} and \eqref{2DimenkernelB}, one can see that the scalar part of the kernel $K_{m,scal}^c, \, m>2$ reduces to
\begin{align*}
\mathcal{L}\left(K_{m,scal}^c\right)
&=\mathcal{L}\left(\frac{1}{2\lambda}z^{-\frac{\mu-2}{2}}A_{\lambda}\right)+\mathcal{L}\left(\frac{1}{2}z^{-\frac{\mu-2}{2}}B_{\lambda}\right)\\
&=(-1)^{n(\lambda +1)}\,z^{2\tilde{n}+2}\,\frac{1}{r}\frac{(r-s)^{n\lambda +\tilde{n}}}{\left((r-s)^n-b_+\right)^{\lambda +1}}\frac{1}{\left((r-s)^n-b_-\right)^{\lambda +1}}\\
&\quad +i\,w\, (-1)^{n(\lambda +1)}\,(-1)^{\tilde{n}}\,z\,\frac{1}{r}\frac{(r-s)^{n\lambda +n+\tilde{n}}}{\left((r-s)^n-b_+\right)^{\lambda +1}}\frac{1}{\left((r-s)^n-b_-\right)^{\lambda +1}}\\
&\quad -i\,w\, (-1)^{n(\lambda +1)}\,z^n\,\frac{1}{r}\,\frac{(r-s)^{n\lambda +n-\tilde{n}}}{\left((r-s)^n-b_+\right)^{\lambda +1}}\frac{1}{\left((r-s)^n-b_-\right)^{\lambda +1}}\\
&\quad +i^{n-1}\, (-1)^{n(\lambda +1)}\,\frac{1}{r}\,\frac{(r-s)^{n\lambda +2n-\tilde{n}}}{\left((r-s)^n-b_+\right)^{\lambda +1}}\frac{1}{\left((r-s)^n-b_-\right)^{\lambda +1}}.
\end{align*}
Adding the bivector part $K_{m,biv}^c=-z^{-\frac{\mu}{2}}C_{\lambda},\, m>2$ in the above theorem and using \eqref{MLinvers}, we thus conclude the results:
\begin{theorem}\label{EqualIntKerInDimEve}
Let $b_{\pm}:=e^{\pm i\theta}i^{ n}z^n$ and 
\begin{align*}
h_{\lambda +1}(t)
&:=\int_0^t\zeta^{n(\lambda +1) -1}E^{\lambda +1}_{n, n(\lambda +1)}\left(b_+\,\zeta^n\right)(t-\zeta)^{n(\lambda +1)-1}E^{\lambda +1}_{n,n(\lambda +1)}(b_-(t-\zeta)^n)\,\mathrm{d}\zeta.
\end{align*}
Then for the dimension $m>2\,(\lambda >0)$ and $1+c=\frac{1}{n}$, $n\in \mathbb{N}_0\backslash\ \{1\}$ odd, $n=2\tilde{n}+1$, $\tilde{n}\in\mathbb{N}_0$, the kernel of the radially deformed Fourier transform $\cF_{D} =   e^{i\frac{\pi}{2}\left( \frac{1}{2}+\frac{m-1}{2(1+c)}\right)}e^{\frac{-i\pi}{4(1+c)^2}\left({\bf D}^2 - (1+c)^2\ux^2\right)}$ takes the form
\[
K_{m}^c =K_{m,scal}^c+ (\ux \wedge \uy)\,K_{m,biv}^c
\]
with
\begin{align*}
K_{m,scal}^c
&=c_3\,z^{n\lambda +3\tilde{n}+2}\bigg(\int_0^1(1+2\tau)^{-(n\lambda+\tilde{n})/2}J_{n\lambda+\tilde{n}}\left(z\sqrt{1+2\tau}\right)h_{\lambda +1}(\tau)\,\mathrm{d}\tau\\
&\quad +i\,w\,(-1)^{\tilde{n}}\int_0^1(1+2\tau)^{-(n\lambda+n+\tilde{n})/2}J_{n\lambda+n+\tilde{n}}\left(z\sqrt{1+2\tau}\right)h_{\lambda +1}(\tau)\,\mathrm{d}\tau\\
&\quad -i\,w\int_0^1(1+2\tau)^{-(n\lambda+\tilde{n}+1)/2}J_{n\lambda+\tilde{n}+1}\left(z\sqrt{1+2\tau}\right)h_{\lambda +1}(\tau)\,\mathrm{d}\tau\\
&\quad +(-1)^{\tilde{n}}\int_0^1(1+2\tau)^{-(n\lambda+3\tilde{n}+2)/2}J_{n\lambda+3\tilde{n}+2}\left(z\sqrt{1+2\tau}\right)h_{\lambda+1}(\tau)\,\mathrm{d}\tau\bigg)
\end{align*}
and
\begin{align*}
K_{m,biv}^c
&=-i\,c_3\,z^{n\lambda+3\tilde{n}+1}\bigg(\int_0^1(1+2\tau)^{-(n\lambda+\tilde{n}+1)/2}J_{n\lambda +\tilde{n}+1}\left(z\sqrt{1+2\tau}\right)h_{\lambda +1}(\tau)\,\mathrm{d}\tau\\
&\quad -(-1)^{\tilde{n}}\int_0^1(1+2\tau)^{-(n\lambda+3\tilde{n}+1)/2}J_{n\lambda +3\tilde{n}+1}\left(z\sqrt{1+2\tau}\right)h_{\lambda +1}(\tau)\,\mathrm{d}\tau\bigg)
\end{align*}
where $w =\cos{\theta}=\frac{\langle x, \, y \rangle}{z}$, $z=|x||y|$ and $c_3=(-1)^{n(\lambda +1)}$.
\end{theorem}
\begin{remark}\label{ConneInteKerM}
If we take the limit $\lambda =0$, the results for dimension 2 in Theorem \ref{InteGralKerDim2} can be reobtained.
\end{remark}
By Theorem \ref{InteGralKerDim2}, Theorem \ref{EqualIntKerInDimEve} and Remark \ref{ConneInteKerM}, the explicit expressions of the kernel of the radially deformed Fourier transform $\cF_{D}$ can be considered simultaneously for all dimensions. 

\section{Conclusions}
The radially deformed Fourier transform is the Fourier transform associated with the radially deformed Dirac operator in the Clifford analysis setting. It was previously studied as one of several function theoretical aspects of the operator. In this paper we developed a new Laplace transform method for the kernel of the deformed Fourier transform  for the case of $1+c=\frac{1}{n}, n\in \mathbb{N}_0\backslash\{1\}$ with $n$ odd. By introducing an auxiliary variable $t$ in the series expansion of the kernel, the Laplace transform of the kernel was obtained. We further simplified these results by making use of the Poisson kernel and the generating function of the Gegenbauer polynomials. Then the rational functions in the kernels were given in partial fraction decomposition when the dimension $m$ is even. Moreover, explicit and integral expressions of the kernel for all dimensions were given using the Laplace inversion of the Mittag-Leffler functions. 

\section*{Acknowledgements}
ZY is supported by the China Scholarship Council (CSC), number 201906120041.


\end{document}